\documentclass{amsart}
\usepackage{amsmath}
\usepackage{amsfonts}
\usepackage{amssymb}
\usepackage{epsfig}
\newtheorem{theorem}{Theorem}[section]

\newtheorem{lm}[theorem]{Lemma}

\newtheorem{cor}[theorem]{Corollary}

\newtheorem{rem}[theorem]{Remark}
\newtheorem{pr}[theorem]{Proposition}
\usepackage{color}

\newtheorem{ex}[theorem]{Example}

\begin{document}

\title{Invariants of $G_2$ and $Spin(7)$ in positive characteristic}
\author{A.N.Zubkov and I.P.Shestakov}
\date{}
\begin{abstract}
The purpose of this paper is to show that over an infinite field of odd characteristic, invariants of $G_2$ and $Spin(7)$, both acting on several copies of octonions, are generated by
the same invariants of degree at most $4$ as in the case of a field of characteristic zero.
\end{abstract}

\maketitle
\section*{Introduction}

Invariants of $G_2$ and $Spin(7)$, both acting on several copies of octonions, have been decribed in \cite{schw2} over a ground field of characteristic zero. The same result was obtained by a different method in \cite{howe}.

In the current manuscript, we extend this result to an arbitrary infinite field of odd characteristic. More precisely, we prove that the corresponding algebras of invariants are generated by the same invariants of degree at most $4$ as in the case of a field of characteristic zero (see Theorem \ref{finaldescription}, Corollary \ref{G_2 inv} and Remark \ref{finalrem} below). 

The article is organized as follows. In the first section, an octonion algebra $\bf O$ is introduced in the form of a Zorn vector-matrix algebra. In the second section, we recall  necessary results about modules with good filtrations and saturated subgroups of reductive groups.

The simple exceptional group $G=G_2$ is isomorphic to the full group of automorphisms of the algebra $\bf O$. The group $G$ acts naturally on a subspace ${\bf O}_0$ of traceless octonions. We prove that all exterior powers $\Lambda^i({\bf O}_0)$ are tilting modules, i.e. they have both good and Weyl filtrations provided the characteristic of the ground field $F$ is either 
zero or odd. Using this result and some specific filtrations of homogeneous components of $F[{\bf O}_0^n]$ by $GL(n)\times G$-submodules, one can prove that the Hilbert-Poincare series of the invariant subalgebra $R(n)=F[{\bf O}_0^n]^G$ does not depend on the characteristic of $F$. The problem of a description of generators of $R(n)$ is equivalent to the problem of a description of generators of the larger algebra $\tilde{R}(n)=F[{\bf O}^n]^G$. In fact, $\tilde{R}(n)\simeq R(n)\otimes F[t(z_1), \ldots , t(z_n)]$, where $z_1, \ldots , z_n$ are generic octonions. 

In the fourth section, we describe systems of (homogeneous) parameters of the algebra $R(n)$, extending Theorem 7.13 of \cite{schw2} to any (perfect) field of odd characteristic. We also prove that the field of rational invariants of the group $G$ is isomorphic to a field of rational functions in $7n-14$ variables (see Corollary \ref{rationality}).

In the fifth section, we use properties of Moufang loops ${\bf O}_1$ and $M={\bf O}_1/\{\pm 1_{\bf O}\}$ to describe affine quotients $SO(7)/G$ and $Spin(7)/G$. This allows us to deduce that $G$ is saturated in both $SO(7)$ and $Spin(7)$. Moreover, arguing in a similar way, one can prove that $SO(7)$ is saturated in both $SO(8)$ and $PSO(8)$, and $Spin(7)$ is saturated in $SO(8)$.

In the sixth section, we show that the algebra $\tilde{R}(n)=F[{\bf O}^n]^G$ is an epimorphic image of $F[{\bf O}^{n+1}]^{Spin(7)}$. Therefore, it is sufficient to find generators of the algebra of spinor invariants. In solving of this problem, the key role is played by Proposition \ref{noname} which is interesting on its own. Indeed, this proposition gives a sufficient condition when a decription of vector invariants of arbitrary reductive group can be reduced to a description of its multilinear invariants. It is likely that this proposition will be useful later, say in a description of generators of vector invariants of other exceptional simple groups. 

The next step is to describe generators of $F[{\bf O}^t]_{1^t}^{Spin(7)}\simeq ({\bf O}^{\otimes t})^{Spin(7)}$. Since $Spin(7)$ contains the matrix $-E=-id_{\bf O}$, the space $({\bf O}^{\otimes t})^{Spin(7)}$ is non-trivial if and only if $t$ is an even integer. Thus $Spin(7)$ acts on $({\bf O}^{\otimes t})^{Spin(7)}$ via the epimorphism $Spin(7)\to Spin(7)/\{\pm E\}\simeq SO(7)^{\rho}$, where the groups $Spin(7)/\{\pm E\}$ and $SO(7)$ are considered as subgroups of $PSO(8)$ and $\rho$ is one of {\it triality automorphisms} of $PSO(8)$. Over a field of characteristic zero, it is well-known that 
${\bf O}^{\otimes 2}$ is isomorphic to $\oplus_{0\leq i\leq 3}\Lambda^i({\bf O}_0)$, where both are regarded as $SO(7)$-modules (cf. \cite{brauerweyl}). We extend this result to an algebraically closed field of odd characteristic. An analogous statement can be proved for a tensor square of a spinor representation of any dimension and we intend to prove this more general result in a forthcoming article.

Further, the $SO(7)$-module ${\bf O}^{\otimes t}$ is isomorphic to the direct summand
\[W_t=\oplus_{0\leq t_1, \ldots , t_k\leq 3}\Lambda^{t_1}({\bf O}_0)\otimes\ldots\otimes\Lambda^{t_k}({\bf O}_0)\]
of an algebraic $SO(7)$-algebra $\Lambda({\bf O}_0^{\oplus k})$, where $t=2k$. Generators of its subalgebra $\Lambda({\bf O}_0^{\oplus k})^{O(7)}$ were described in \cite{adamryb}. Using an analogous approach, we describe the subalgebra $\Lambda({\bf O}_0^{\oplus k})^{SO(7)}$.

Unfortunately, it is not easy to describe generators of the vector space $({\bf O}^{\otimes t})^{Spin(7)}$ using the generators of the vector space $W_t$. We do not see how the above-mentioned isomorphism between $({\bf O}^{\otimes t})^{Spin(7)}$ and $W_t$ can be used to prove that spinor invariants are generated by invariants of degrees $2$ and $4$.

In the seventh section, instead of a direct computation, we offer an elegant procedure that allows us to prove that the image of any invariant $w$ from $W_t$, regarded as a multilinear polynomial, can be obtained from the image of an invariant $w'$ from $W_{t+2}$, whenever some parameter of $w$, called the {\it decomposition index}, is not minimal. More precisely, the image of $w$ is obtained by an operation, called {\it convolution}, between the image of $w'$ and some polynomial of degree $6$. Since the decomposition index of $w'$ is less than the decomposition index of $w$, one can use an induction on the decomposition index to conclude that the spinor invariants are generated by the invariants of degrees 2 and 4 (see Proposition \ref{keyprop2} and Lemma \ref{convolutionofstandardpolynomials}).

Finally, there are only finitely many elements $w$ of minimal index. We prove that the images of all such elements can be obtained from the image of only one among them using a convolution with an element of degree $6$.

\section{Octonions and generic octonions}

The content of this section can be found in \cite{zsss}, chapter 2.

Let $F$ be an infinite field of arbitrary characteristic. The {\it Cayley algebra} ${\bf O}(F)$ is a {\it Zorn vector-matrix algebra} consisting of
all matrices
$$a=\left(\begin{array}{cc}
\alpha & {\bf u} \\
{\bf v} & \beta	
\end{array}\right), \text{ where } \alpha, \beta\in F \text{ and } {\bf u}, {\bf v}\in F^3 , 
$$ 
with the multiplication given by the rule
$$\left(\begin{array}{cc}
\alpha & {\bf u} \\
{\bf v} & \beta	
\end{array}\right) \left(\begin{array}{cc}
\alpha' & {\bf u'} \\
{\bf v'} & \beta'	
\end{array}\right) =\left(\begin{array}{cc}
\alpha\alpha' +{\bf u}\cdot {\bf v'} & \alpha{\bf u'} +\beta'{\bf u} -{\bf v}\times {\bf v'} \\
\alpha'{\bf v} + \beta{\bf v'} + {\bf u}\times {\bf u'} & \beta\beta' +{\bf v}\cdot {\bf u'}	
\end{array}\right),$$
where 
\[{\bf u}\cdot {\bf v}=u_1v_1 +u_2v_2 +u_3v_3 \text{ and } {\bf u}\times {\bf v}=(u_2v_3-u_3v_2, u_3v_1-u_1v_3, u_1v_2-u_2v_1).\]
The elements of ${\bf O}(F)$ are called {\it octonions} and ${\bf O}(F)$ is also called the {\it octonion algebra} (over $F$). 
If it does not lead to confusion, we denote ${\bf O}(F)$ just by $\bf O$. For any field extension $L\subseteq F$, the $L$-algebra ${\bf O}(L)$ is naturally identified with an $L$-subalgebra of ${\bf O}(F)$.

The algebra $\bf O$ is a simple alternative algebra with the multiplicative {\it quadratic form} (or {\it norm})
$$n : \left(\begin{array}{cc}
\alpha & {\bf u} \\
{\bf v} & \beta	
\end{array}\right)\to \alpha\beta - {\bf u}\cdot {\bf v}.$$
Also, $\bf O$ admits an algebra involution $$a\to \overline{a}=\left(\begin{array}{cc}
\beta & -{\bf u} \\
-{\bf v} & \alpha	
\end{array}\right)$$
such that $n(a)=a\overline{a}$. Denote $a+\overline{a}$ by $t(a)$ and call it the {\it trace} of $a$. Notice that $t(a)\in F 1_{\bf O}\simeq F$ and $a$ satisfies the quadratic equation
$a^2-t(a)a+n(a)=0$. Denote by $q( \ , \ )$ the bilinear form on $\bf O$ associated with the norm $n$.
By definition, $q(a, b)=a\overline{b}+b\overline{a}$ for $a, b\in {\bf O}$. In the matrix form, we have
$$q(\left(\begin{array}{cc}
\alpha & {\bf u} \\
{\bf v} & \beta	
\end{array}\right), \left(\begin{array}{cc}
\alpha' & {\bf u'} \\
{\bf v'} & \beta'	
\end{array}\right))=\alpha\beta' +\alpha'\beta -{\bf u}\cdot {\bf v'}-{\bf u'}\cdot {\bf v}.$$
Denote $${\bf c}_1=(1,0,0), {\bf c}_2=(0,1,0), {\bf c}_3=(0,0,1)$$
and
$${\bf u}_i=\left(\begin{array}{cc}
0 & {\bf c}_i \\
{\bf 0} & 0	
\end{array}\right), {\bf v}_i=\left(\begin{array}{cc}
0 & {\bf 0} \\
{\bf c}_i & 0	
\end{array}\right) \text{ for } i=1, 2, 3.
$$  
The idempotents
$$
\left(\begin{array}{cc}
1 & {\bf 0} \\
{\bf 0} & 0	
\end{array}\right) \text{ and } \left(\begin{array}{cc}
0 & {\bf 0} \\
{\bf 0} & 1	
\end{array}\right)$$
are denoted by $e_1$ and $e_2$, respectively, and the unit $e_1 +e_2$ of $\bf O$ is denoted by $1_{\bf O}$. To simplify our notations, we identify
the matrices
$$\left(\begin{array}{cc}
0 & {\bf u} \\
{\bf 0} & 0	
\end{array}\right), \left(\begin{array}{cc}
0 & {\bf 0} \\
{\bf v} & 0	
\end{array}\right)$$
with the vectors ${\bf u}, {\bf v}\in F^3$. In particular, an octonion $a$ can be written as
$\alpha e_1 +\beta e_2 +{\bf u}+{\bf v}$.
The following relations are now obvious:
$${\bf u}e_1=e_2{\bf u}={\bf v}e_2=e_1{\bf v}={\bf 0}, e_1{\bf u}={\bf u}e_2={\bf u},
{\bf v}e_1=e_2{\bf v}={\bf v},$$
$${\bf u}^2={\bf 0}, {\bf u}_i{\bf u}_j=(-1)^{\epsilon_{ij}}{\bf v}_k,
{\bf v}^2={\bf 0}, {\bf v}_i{\bf v}_j=(-1)^{\epsilon_{ji}}{\bf u}_k ,$$
where $i\neq j, i\neq k, j\neq k$, and $\epsilon_{ij}$ is the parity of the substitution
$$\left(\begin{array}{ccc}
1 & 2 & 3 \\
k & i & j	
\end{array}\right).$$ 
We also have
$${\bf u}{\bf v}=({\bf u}\cdot {\bf v}) e_1 \text{ and } {\bf v}{\bf u}=({\bf v}\cdot {\bf u}) e_2 .$$
The basis consisting of elements $e_1, e_2, {\bf u}_i, {\bf v}_i$ for $1\leq i\leq 3$ is called the {\it standard basis} of ${\bf O}$. 
Denote by ${\bf O}_0$ the subspace $\{a\in {\bf O}| t(a)=0\}$ of all traceless octonions.

Let $V$ be a vector space. 
Fix a basis $v_1, \ldots , v_s$ of $V$. A coordinate function $z_{ij}$ on 
$V^n=\underbrace{V\oplus\ldots\oplus V}_{n-\mbox{times}}$ is defined as $z_{ij}((v^{(1)}, \ldots , v^{(n)}))=c_{ij}$, where
$v^{(i)}=\sum_{1\leq t\leq s}c_{it}v_t$ for $1\leq i\leq n, 1\leq j\leq s$. The vector $z_i= \sum_{1\leq j\leq s}z_{ij}v_j$ is called the $i$-th {\it generic vector} related to the space $V^n$.

If $G$ is an algebraic group and $V$ is a (rational) $G$-module, then $G$ acts on $V^n$ diagonally. This induces a $G$-action on $F[V^n]=F[z_{ij}]$ by $(g\cdot f)(w)=f(g^{-1}w)$
for $f\in F[V^n]$ and $w\in V^n$. If the coordinate functions are assumed to be of degree 1, then the algebra $F[V^n]$ is
$\mathbb{N}^n$-graded and $\mathbb{N}$-graded as well, say 
\[F[V^n]=\oplus_{k_1,\ldots, k_n\geq 0} F[V^n]_{k_1,\ldots, k_n} \text{ and } F[V^n]=\oplus_{k\geq 0}F[V^n]_k,\]
where
\[F[V^n]_k=\oplus_{k_1+\ldots +k_n=k}F[V^n]_{k_1,\ldots, k_n}.\]

If $V$ is a non-degenerate quadratic space with the associated bilinear form $q$, 
then there is a natural isomorphism $V\to V^*$ of $O(q)$-modules, given by $v\mapsto v^*$, where $v^*(w)=q(v, w)$ for $v, w\in V$. In particular, $F[V^n]$ is isomorphic to $S(V^n)$ as a graded algebraic $O(q)$-algebra.

The $i$-th generic vector, related to the space ${\bf O}^n$, is denoted by $z_i$ and it is called the $i$-th {\it generic octonion}.
Correspondingly, the $i$-th generic vector, related to the space ${\bf O}_0^n$, is denoted by $x_i$ and it is called the $i$-th {\it traceless generic octonion}. 
\begin{ex}\label{anisom}
Let $(V, q)=({\bf O}, q)$ and denote by $e_1, \ldots , e_8$ the above basic elements such that $e_{i+2}={\bf u}_i$ and $e_{i+5}={\bf v}_i$ for $1\leq i\leq 3$. 
For $1\leq i\leq n$ and $1\leq j\leq 8$ denote by $e_{ij}$ the element $(0,\ldots , \underbrace{e_j}_{i-\mbox{th place}}, \ldots, 0)\in {\bf O}^n$. 
Then the isomorphism $S({\bf O}^n)\simeq F[{\bf O}^n]$ of $O(q)$-algebraic algebras is given by $e_{i 1}\mapsto z_{i 2}, e_{i 2}\mapsto z_{i 1}, e_{i j}\mapsto -z_{i, j+3}, e_{i k} \mapsto -z_{i, k-3}$ for $1\leq i\leq n, 3\leq j\leq 5$ and $6\leq k\leq 8$.
\end{ex}
\begin{ex}\label{anisomfortraceless}
Let $(V, q)=({\bf O}_0 , q|_{{\bf O}_0})$. The vectors $a_1=e=e_1-e_2, a_{i+1}={\bf u}_i, a_{i+4}={\bf v}_i$ for $1\leq i\leq 3$ form a basis of ${\bf O}_0$. 
Analogously to $e_{ij}$ of Example \ref{anisom}, one can define the basic vectors $a_{i j}\in {\bf O}_0^n$. Then the isomorphism $S({\bf O}_0^n)\to F[{\bf O}_0^n]$ of $O(q|_{{\bf O}_0})$-algebraic algebras is given by $a_{i1}\mapsto -2x_{i 1}, a_{i j}\mapsto -x_{i, j+3}, a_{i k}\mapsto -x_{i, k-3}$ for $1\leq i\leq n, 2\leq j\leq 4$ and $5\leq k\leq 7$.
\end{ex}

\section{Modules with good and Weyl filtrations} 

For the content of this section we refer to \cite{don} and \cite{jan}, part II. In this section, the base field $F$ is assumed to be algebraically closed.

Let $G$ be an algebraic group. Choose a Borel subgroup $B^+$ and a maximal torus $T\leq B^+$. Let $\Phi$ be a set of roots of $G$ such that
the elements of $\Phi^+$ are weights of $Lie(B^+)$. Denote by $B^{-}$ the Borel subgroup of $G$ {\it opposite} to $B^+$. The unipotent radicals of $B^+$ and $B^-$ are denoted by $U^+$ and $U^-$ respectively. 

For $\lambda\in X(T)$ define the {\it induced} $G$-module $$H^0(\lambda)=ind^G_{B^-} F_{\lambda}=\{f\in F[G]| f(bg)=\lambda(b) f(g), g\in G, b\in B^{-}\},$$ where $F_{\lambda}$ is a one-dimensional  $B^-$-module of weight $\lambda$. The group $G$ acts on $H^0(\lambda)$ as $(gf)(x)=f(xg)$ for $x, g\in G$ and $f\in F[G]$. 

In what follows, we assume that $G$ is reductive (for the general setting see \cite{don}). 
Recall that $H^0(\lambda)\neq 0$ if and only if $\lambda$ belongs to the set of dominant weights $X(T)^+$. The set $X(T)^+$ is partially ordered in such a way that $\lambda\leq\mu$ if and only if $\mu-\lambda$ is a sum of positive roots. 

The socle of $H^0(\lambda)$ is a simple $G$-module
$L(\lambda)$ and all other simple sections of $H^0(\lambda)$ are isomorphic to $L(\mu)$, where $\mu <\lambda$. 
Moreover, any $G$-module which satisfies these two conditions is a submodule of $H^0(\lambda)$. In other words, every $H^0(\lambda)$ is a costandard object in the highest weight category of rational $G$-modules (cf. \cite{cps, dr}).

Let $w_0$ denote the longest element of Weyl group $W(G, T)=N_G(T)/T$. 
The $G$-module $V(\lambda)=H^0(\lambda^*)^*$, where $\lambda^*=-w_0(\lambda)$, is called
a {\it Weyl} module. Dually, $V(\lambda)$ has the simple top $L(\lambda)$ and all sections of its radical are isomorphic to $L(\mu)$, where $\mu <\lambda$. As above, any $G$-module which satisfies these two conditions is an epimorphic image of $V(\lambda)$. Thus $V(\lambda)$ is a standard object in the highest weight category of rational $G$-modules. 

We say that a $G$-module $W$ has a {\it good filtration} if there is a filtration with at most countably number of members
$$0=W_0\subseteq W_1\subseteq W_2\subseteq\ldots, \ \bigcup_{i\geq 0}W_i
=W$$
such that $W_i/W_{i-1}\simeq H^0(\lambda_i)$ for every $i\geq 1$.  Dually, a filtration such that $W_i/W_{i-1}\simeq V(\lambda_i)$ for every $i\geq 1$ is called a {\it Weyl filtration}. If $W$ is finite-dimensional and $W$ has both a good filtration and a Weyl filtration, then
$W$ is called a {\it tilting} $G$-module.

The property that a module has a good or Weyl filtration does not depend on the choice of subgroups $B^+$ and $B^-$ (see \cite{don}, pp. 10, 31). 

The next proposition summarizes important properties of modules with good (or Weyl) filtrations.

\begin{pr}\label{propofgoodandWeyl} The following properties hold:
\begin{enumerate}
\item Let $V$ be a $G$-module and $W$ be its submodule. If both $V$ and $W$ are $G$-modules with good filtration, then $V/W$ has also a good filtration. Additionally, every direct summand of $V$ is a $G$-module with good filtration.
\item A finite-dimensional $G$-module $V$ has a good filtration if and only if $V^*$ has a Weyl filtration. In particular, $V$ is tilting if and only if 
$V$ and $V^*$ are $G$-modules with good filtration if and only if $V$ and $V^*$ are $G$-modules with Weyl filtration.    
\item If a $G$-module $V$ has a good filtration, then the multiplicity of $H^0(\lambda)$ in any such filtration is equal to
$c_{\lambda}(V)=\dim Hom_G(V(\lambda), V)$. Dualizing, if $V$ has a Weyl filtration, then the multiplicity of $V(\lambda)$ in any such filtration is 
equal to $d_{\lambda}(V)=\dim Hom_G(V, H^0(\lambda))$. 
\item For every $\lambda\in X(T)^+$ the {\it formal characters} $\chi(H^0(\lambda))$ and $\chi(V(\lambda))$ are equal to each other 
and do not depend on $char F$. Moreover, the formal characters of costandard/standard modules are linearly independent.
In particular, if $V$ has a good or Weyl filtration, then  
$$\chi(V)=\sum_{\lambda\in X(T)^+}c_{\lambda}(V)
\chi(H^0(\lambda))$$$$(\mbox{and } \ \chi(V)=\sum_{\lambda\in X(T)^+}d_{\lambda}(V)
\chi(V(\lambda)), \mbox{ respectively}),$$ where the coefficients $c_{\lambda}(V)$ (and $d_{\lambda}(V)$, respectively) are uniquely defined by $\chi(V)$.  
\item If $V$ and $W$ are $G$-modules with good (or Weyl) filtration, then $V\otimes W$ has a good (or Weyl)
filtration with respect to the diagonal action of $G$.
\item If \ $0\to V\to W\to U\to 0$ is a short exact sequence of $G$-modules and $V$ has a good filtration, then
$0\to V^G\to W^G\to U^G\to 0$ is also exact.
\end{enumerate}
\end{pr}

In what follows a tensor, symmetric or exterior power of a $G$-module $V$ is a $G$-module considered with respect to the diagonal action.   

A reductive subgroup $H$ of $G$ is said to be {\it good} or {\it saturated} if and only if every $G$-module $V$ with good filtration is also an $H$-module
with good filtration. The following proposition is proved in \cite{don2}, section 1.4. 
\begin{pr}\label{goodsubgroup}
A subgroup $H$ of $G$ is good if and only if $Ind^G_H F=F[H\backslash G]\simeq F[G/H]$ is a left $G$-module with good filtration.
\end{pr}
\begin{ex}\label{orthogonal} (see \cite{z})
Let $F$ be an algebraically closed field of characteristic different from 2, $V$ be a non-degenerate quadratic $F$-space, and  
$SO(V)=SO(n)$ be the corresponding special orthogonal group.
Every exterior power $\Lambda^i(V)$ is a direct sum of at most two indecomposable tilting submodules.
In fact, $\Lambda^{n-i}(V)^*\simeq\Lambda^i(V)\simeq\Lambda^i(V^*)\simeq\Lambda^i(V)^*$, where $n=\dim V$.
If $n=2l+1$ is odd, that is $SO(V)$ is of type $B_l$, then for every $i < l$ we have $\Lambda^i(V)=L(\lambda_i)=H^0(\lambda_i)=V(\lambda_i)$ and $\Lambda^l(V)=L(2\lambda_l)=H^0(2\lambda_l)=V(2\lambda_l)$, where $\lambda_i$ is the $i$-th fundamental dominant weight. Analogously, if $n=2l$, that is $SO(V)$ is of type $D_l$, then
$\Lambda^i(V)=L(\lambda_i)=H^0(\lambda_i)=V(\lambda_i)$ for $i< l-1$, $\Lambda^{l-1}(V)=L(\lambda_{l-1}+\lambda_l)=H^0(\lambda_{l-1}+\lambda_l)=
V(\lambda_{l-1}+\lambda_l)$ and $\Lambda^l(V)=L(2\lambda_{l-1})\oplus L(2\lambda_l)$. Additionally, $L(2\lambda_{l-1})=H^0(2\lambda_{l-1})=
V(2\lambda_{l-1})$ and $L(2\lambda_l)=H^0(2\lambda_l)=V(2\lambda_l)$.   
\end{ex} 
\begin{rem}\label{saturated} (see also \cite{don}, 1.1.14) 
If $H$ is saturated in $G$ and $\sigma\in Aut(G)$, then $H^{\sigma}$ is also saturated in G. 
Also, $G/H^{\sigma}\simeq G/H$ and $G$ acts on the last variety as $g\cdot (xH)=g^{\sigma^{-1}}xH$ for $g, x\in G$. In the notation of \cite{don}, 1.1.14, $F[G/H^{\sigma}]\simeq F[G/H]^{\sigma^{-1}}$ has a good filtration.
\end{rem}
\begin{rem}\label{extensions} (see also \cite{jan}, part II, 2.14)
Let $V$ be a $G$-module. If $Ext_G^1(V, H^0(\lambda))\neq 0$, then $V$ has a composition factor $L(\mu)$ such that
$\mu >\lambda$. In fact, by \cite{jan}, part II, Proposition 4.18, the injective hull $I(\lambda)$ of $L(\lambda)$ has a good filtration 
whose first factor is $H^0(\lambda)$ and all other factors are $H^0(\mu)$ with $\mu>\lambda$. A fragment of long exact sequence 
$$Hom_G(V, I(\lambda)/H^0(\lambda))\to Ext_G^1(V, H^0(\lambda))\to Ext_G^1(V, I(\lambda))=0$$
shows that there is a finite-dimensional submodule $W\subseteq V$ and $\mu >\lambda$ such that $Hom_G(W, H^0(\mu))\neq 0$. In particular,
if $V$ has a good fltration, then there is a filtration $\{V_i\}_{i\geq 0}$ with factors $V_i/V_{i-1}\simeq H^0(\lambda_i)$ such that $i\leq j$ implies either 
$\lambda_i\leq\lambda_j$ or $\lambda_i$ and $\lambda_j$ are incomporable. An analogous statement is valid for modules with Weyl filtration.     
\end{rem}
\begin{rem}\label{splitness} 
Let $0\to V\to W\to U\to 0$ be an exact sequence of $G$-modules. Assume that either $V$ is tilting and $U$ has a Weyl filtration, or $V$ has a good filtration and $U$ is tilting. Then by \cite{jan},  Proposition II.4.13, this sequence splits.
\end{rem}
\begin{rem}\label{wheninducedistilting}
Observe that $L(\lambda)=H^0(\lambda)$ if and only if $L(\lambda)=V(\lambda)$ if and only if $V(\lambda)=H^0(\lambda)$. 
\end{rem}

\section{The group $G_2$}

In this section we again assume that $F$ is algebraically closed.

The full group of automorphisms $Aut({\bf O})$ is known to be isomorphic to the Chevalley group $G=G_2(F)$, (cf.
\cite{car}). The group $G$ contains a closed subgroup isomorphic to $SL_3(F)$. In fact, every $g\in SL_3(F)$
acts as an automorphism on $\bf O$ by the rule $e_i\mapsto e_i$ for $i=1, 2$ and ${\bf u}\mapsto {\bf u}g, {\bf v}\mapsto {\bf v}(g^{-1})^t$. 
Choose a maximal torus $T$ of $G$ to coincide with the standard maximal torus of the subgroup $SL_3(F)$.

For every ${\bf u}, {\bf v}\in {\bf O}$ we define two automorphisms
$\delta({\bf u})$ and $\delta({\bf v})$ acting as
$$e_1\mapsto e_1 +{\bf u}, e_2\mapsto e_2 -{\bf u}, {\bf u}'\mapsto {\bf u}' -{\bf u}'{\bf u}, {\bf v}'\mapsto {\bf v}' -({\bf u}\cdot {\bf v}')e -({\bf u}\cdot {\bf v}'){\bf u},$$
and 
$$e_1\mapsto e_1 -{\bf v}, e_2\mapsto e_2 +{\bf v}, {\bf u}'\mapsto {\bf u}' +({\bf u}'\cdot {\bf v})e -({\bf u}'\cdot {\bf v}){\bf v},
{\bf v}'\mapsto {\bf v}' -{\bf v}'{\bf v},$$
respectively. 
Let
$$\Phi =\{\pm\omega_1, \pm\omega_2, \pm\omega_3, \pm(\omega_1 -\omega_2), \pm(\omega_2 -\omega_3), \pm(\omega_3 -\omega_1) |
\omega_1 +\omega_2 +\omega_3 =0\}$$
be a root system of type $G_2$. The corresponding root subgroups are
$$X_{\omega_i}(t)=\delta(t{\bf u}_i), X_{-\omega_i}(t)=\delta(t{\bf v}_i) \text{ and }
X_{\omega_i -\omega_j}(t)=E+t E_{ji} \text{ for } 1\leq i\neq j\leq 3,  
$$ 
where $t\in F$, $E$ is the identity matrix of size $3\times 3$ and $E_{ij}$ are matrix units of size $3\times 3$, i.e. ${\bf u}_k E_{ij}=\delta_{k i}{\bf u}_j$ for 
$1\leq i\neq j, k\leq 3$. 
The above root subgroups generate $G$.

Define the short and long fundamental roots
to be $\alpha=\omega_2$ and $\beta=\omega_1-\omega_2$. Then the system of positive roots $\Phi^{+}$ is
$$\omega_1=\alpha +\beta, \omega_2=\alpha, -\omega_3=2\alpha+\beta,$$
$$\omega_1-\omega_3=3\alpha +2\beta, \omega_1-\omega_2=\beta \text{ and } \omega_2-\omega_3=3\alpha+\beta.
$$
The {\it fundamental dominant} weights are $$\lambda_1=3\alpha+2\beta=\omega_1-\omega_3 \text{ and } \lambda_2=2\alpha +\beta=-\omega_3$$ and every {\it dominant} weight $\lambda$
equals $c_1\lambda_1 +c_2\lambda_2$ for  $c_1, c_2\geq 0$ (cf. \cite{ham}, Appendix). 

It can be easily verified that for two dominant weights 
$\lambda=c_1\lambda_1 +c_2\lambda_2$ and $\mu=c'_1\lambda_1 +c'_2\lambda_2$ the condition $\lambda\geq\mu$ is equivalent to $3(c_1-c'_1) \geq -2(c_2-c'_2)$. For example, 
$\lambda_1\geq \lambda_2$.

It is easy to see that $G$ commutes with the involution. In particular, $t(ga)=t(a), n(ga)=n(a)$ and $q(ga, gb)=q(a, b)$ for $a, b\in {\bf O}$ and $g\in G$.
There is an orthogonal decomposition ${\bf O}=F\cdot 1_{\bf O} \ \bot \ {}\bf O_0$ with respect to the form $q$. 
It is clear that $G{\bf O}_0={\bf O}_0$, that is ${\bf O}_0$ is a faithful representation of $G$. 
\begin{lm}\label{zeronorm}
If the norm of a (non-zero) traceless octonion $x$ equals zero, then there is an element $g\in G$ such that $gx={\bf u}_1$.
\end{lm}
\begin{proof} Let $x=\alpha e +{\bf u}+{\bf v}$. If $\alpha\neq 0$, then $({\bf u}\cdot {\bf v})=-\alpha^2$ and 
$\delta(-\frac{1}{\alpha}{\bf u})(x)={\bf u}'+{\bf v}'$. If $\alpha=0$, 
then a suitable element $g\in SL_3(F)$ takes $x$ to ${\bf u}_1 , -{\bf v}_1$ or ${\bf u}_1 +\beta{\bf v}_2$, where $\beta\neq 0$. The elements ${\bf u}_1$ and $-{\bf v}_1$ are conjugated by the automorphism
$h\in G$ such that $h^2=id_A$, $he_1=e_2, he_1=e_2$ and $h{\bf u}_i=-{\bf v}_i$ for $1\leq i\leq 3$. Finally,
$\delta(-\beta{\bf u}_3)({\bf u}_1 +\beta{\bf v}_2)={\bf u}_1$.\end{proof}

\begin{pr}\label{exterioraretilting}
If $char F\neq 2$, then $\Lambda^k({\bf O}_0)\simeq\Lambda^k({\bf O}^*_0)$ is a tilting \ $G$-module for every non-negative integer $k$. 
\end{pr}
\begin{proof} We have ${\bf O}_0 =L(\lambda_2)=H^0(\lambda_2)$ and ${\bf O}_0\simeq {\bf O}_0^*$. In fact, if $L(\mu)$ is a simple submodule
of ${\bf O}_0$, then its {\it highest weight} (or {\it maximal}) vector $v$ is $U^+$-invariant. The only dominant weights of ${\bf O}_0$ are
$0$ and $\lambda_2$. Since $X_{\omega_1}(1)e\neq e$ (the last inequality holds if and only if $char F\neq 2$), we obtain that  
$L(\lambda_2)\subseteq {\bf O}_0$ and ${\bf v}_3$ is a maximal vector of $L(\lambda_2)$. Lemma \ref{zeronorm} implies that $L(\lambda_2)$ contains all ${\bf u}_i$ and ${\bf v}_i$. Combining with $X_{\omega_1}(1){\bf v}_1=-e- {\bf u}_1 +{\bf v}_1$ we see that ${\bf O}_0=L(\lambda_2)$. If $char F=0$, then ${\bf O}_0=L(\lambda_2)=H^0(\lambda_2)=V(\lambda_2)$.
By Proposition \ref{propofgoodandWeyl} (4), the dimension of $H^0(\lambda_2)$ equals seven over a field of any characteristic, 
hence ${\bf O}_0=L(\lambda_2)=H^0(\lambda_2)=V(\lambda_2)$ over any field of odd characteristic.

The second statement holds by $G\leq SO(q|_{{\bf O}_0})=
SO(7)$. Observe that for any non-negative integer $k$ the $G$-module ${\bf O}_0^{\otimes k}$ is tilting.

Using isomorphisms from Example \ref{orthogonal} it remains to prove that $\Lambda^2({\bf O}_0)$ and $\Lambda^3({\bf O}_0)$ are $G$-modules with good filtration.
Since $char F\neq 2$, there is a natural decomposition ${\bf O}_0^{\otimes 2}=S^2({\bf O}_0)\oplus \Lambda^2({\bf O}_0)$. Thus both $G$-modules $\Lambda^2({\bf O}_0)$ and $S^2({\bf O}_0)$ are tilting. If $char F\neq 3$, then $\Lambda^3({\bf O}_0)$ is a direct summand of ${\bf O}_0^{\otimes 3}$, hence tilting. From now on we assume that $char F=3$.

Again, Proposition \ref{propofgoodandWeyl} (4) combined with Proposition 3.22  from \cite{schw2} implies that $S^2({\bf O}_0)$ has a good filtration with factors $H^0(0)$ and $S^2({\bf O}_0)/H^0(0)\simeq H^0(2\lambda_2)$.
It remains to prove that $H^0(2\lambda_2)=L(2\lambda_2)$. In fact, by the same Proposition 3.22 from \cite{schw2}, $\chi(\Lambda^3({\bf O}_0))=\sum_{0\leq i\leq 2}\chi(H^0(i\lambda_2))=\sum_{0\leq i\leq 2}\chi(L(i\lambda_2))$. Thus $\Lambda^3({\bf O}_0)$ has a composition series with factors $L(0), L(\lambda_2)$ and $L(2\lambda_2)$ only, hence it has a good filtration.

The module $S^2({\bf O}_0)$ has a Weyl filtration with factors $V(2\lambda_2)$ and $S^2({\bf O}_0)/V(2\lambda_2)\simeq V(0)$. If $H^0(0)\bigcap V(2\lambda_2)=0$, then $V(2\lambda_2)\simeq H^0(2\lambda_2)$ and $H^0(2\lambda_2)=L(2\lambda_2)$. On the contrary, assume that $H^0(0)\subseteq V(2\lambda_2)$. Then $V(2\lambda_2)/H^0(0)\simeq L(2\lambda_2)$ and $H^0(2\lambda_2)/L(2\lambda_2)\simeq L(0)$.    

Since $n(x_1)\in F[{\bf O}_0]^G$, the isomorphism from Example \ref{anisomfortraceless} takes $n(x_1)$  to $-\frac{1}{4}e^2-\sum_{1\leq i\leq 3}{\bf u}_i{\bf v}_i$ that spans $S^2({\bf O}_0)^G=H^0(0)=L(0)$ (this is valid in any characteristic). 

The vector ${\bf v}_3^2$ is maximal in $H^0(2\lambda_2)$, hence it generates the socle $L(2\lambda_2)$. Using the automorphism $h$ from Lemma \ref{zeronorm} and all automorphisms from
$SL_3(F)$, one can show that $L(2\lambda_2)$ contains all ${\bf u}_i{\bf u}_j$ and ${\bf v}_i {\bf v}_j$. Further, applying the automorphisms
$X_{\pm\omega_i}$ we obtain that all ${\bf u}_i{\bf v}_j$ belong to $L(2\lambda_2)$, provided $i\neq j$.

Let $z=\sum_{1\leq i\leq 3}\alpha_i {\bf u}_i{\bf v}_i$ be a generator of $W=H^0(2\lambda_2)/L(2\lambda_2)$. If $\alpha_i\neq 0$, then $X_{\omega_i}(1)z -z\in L(2\lambda_2)$ implies 
$e{\bf u}_i\in L(2\lambda_2)$; hence all $e{\bf u}_j, e{\bf v}_j$ belong to $L(2\lambda_2)$. For every $i$ we have 
$$X_{-\omega_i}(1)(e{\bf u}_i)=(e-2{\bf v}_i)({\bf u}_i +e-{\bf v}_i)= (e^2-2{\bf u}_i{\bf v}_i)-e{\bf u}_i-e{\bf v}_i -2{\bf v}_i^2, 
$$
and that implies $e^2-2{\bf u}_i{\bf v}_i\in L(2\lambda_2)$.
In other words, all ${\bf u}_i{\bf v}_i$ are equal to $\frac{1}{2}e^2$ modulo $L(2\lambda_2)$. Since $-\frac{1}{4}e^2-\sum_{1\leq i\leq 3}{\bf u}_i{\bf v}_i\in H^0(0)$, one derives
that $W=0$. This contradiction concludes the proof.\end{proof} 
\begin{rem}\label{whichchar?}
The last argument of the proof of Proposition \ref{exterioraretilting} works provided $char F\neq 7$. Therefere $S^2({\bf O}_0)=H^0(0)\oplus H^0(2\lambda_2)$ provided $char F\neq 7$.
\end{rem}
\begin{pr}\label{structureofsecondexterior}
The $G$-module $\Lambda^2({\bf O}_0)$ has the following properties.
\begin{enumerate}
\item If $char F\neq 3$, then $\Lambda^2({\bf O}_0)\simeq H^0(\lambda_2)\oplus H^0(\lambda_1)$. In particular, $H^0(\lambda_1)=V(\lambda_1)=L(\lambda_1)$;
\item If $char F=3$, then $\Lambda^2({\bf O}_0)$ is indecomposable. Moreover, the composition series for $\Lambda^2({\bf O}_0), H^0(\lambda_1)$ and $V(\lambda_1)$ are 
$$\Lambda^2({\bf O}_0)=\begin{array}{c}
L(\lambda_2) \\
| \\
L(\lambda_1) \\
| \\
L(\lambda_2)
\end{array}, H^0(\lambda_1)=\begin{array}{c}
L(\lambda_2) \\
| \\
L(\lambda_1) \\
\end{array} \text { and } V(\lambda_1)=\begin{array}{c}
L(\lambda_1) \\
| \\
L(\lambda_2)
\end{array},$$ respectively. 
\end{enumerate}
\end{pr}
\begin{proof}
Arguing as in Proposition \ref{exterioraretilting} we obtain that $\Lambda^2({\bf O}_0)$ has a good filtration with factors $H^0(\lambda_2)$ and $\Lambda^2({\bf O}_0)/H^0(\lambda_2)\simeq H^0(\lambda_1)$. Symmetrically, $\Lambda^2({\bf O}_0)$ has a Weyl filtration with factors
$V(\lambda_1)$ and $\Lambda^2({\bf O}_0)/V(\lambda_1)\simeq V(\lambda_2)$. The only dominant weights of $\Lambda^2({\bf O}_0)$ are $0, \lambda_1$ and $\lambda_2$.

There is a unique (up to a non-zero scalar from $F$) $U^+$-invariant vector $x=e\wedge {\bf v}_3 -2{\bf u}_1\wedge {\bf u}_2$ of weight $\lambda_2$. It follows that $H^0(\lambda_2)=L(\lambda_2)=Gx$. Furthemore, it can be easily checked that the elements 
$$e\wedge {\bf u}_i{\bf u}_j-2{\bf u}_i\wedge {\bf u}_j, \ e\wedge {\bf v}_i{\bf v}_j+2{\bf v}_i\wedge {\bf v}_j , 1\leq i< j\leq 3 \text{ and } 
\sum_{1\leq i\leq 3}{\bf u}_i\wedge {\bf v}_i$$ 
form a basis of $H^0(\lambda_2)$.

Recall that $Hom_G(V(\lambda), H^0(\mu))\neq 0$ if and only if $\lambda=\mu$. Thus $\Lambda^2({\bf O}_0)\simeq H^0(\lambda_2)\oplus H^0(\lambda_1)$ if and only if $H^0(\lambda_2)\bigcap V(\lambda_1)=0$. Since $\Lambda^2({\bf O}_0)_{\lambda_1}=F{\bf u}_1\wedge {\bf v}_3$, the $U^+$-invariant vector $y={\bf u}_1\wedge {\bf v}_3$ generates $V(\lambda_1)$. 

If $H^0(\lambda_2)\subseteq V(\lambda_1)$, then $V(\lambda_1)/H^0(\lambda_2)\simeq L(\lambda_1)$ and $H^0(\lambda_1)/L(\lambda_1)\simeq H^0(\lambda_2)$. In this case, $H^0(\lambda_2)$ coincides with the socle of $\Lambda^2({\bf O}_0)$, hence $\Lambda^2({\bf O}_0)$ is indecomposable.

Finally,  $U^- y\bigcap H^0(\lambda_2)\neq 0$ if and only if $char F=3$.
In fact, the element $X_{-\omega_1}(1)y$ has a summand ${\bf u}_1\wedge {\bf u}_2+e\wedge {\bf v}_3$ of weight $\lambda_2$ that belongs to $H^0(\lambda_2)$ if and only if $char F=3$.
\end{proof}
\begin{pr}\label{secondsymmetricifchar7}
If $char F=7$, then $S^2({\bf O}_0)$ is indecomposable. Moreover, the composition series for $S^2({\bf O}_0), H^0(2\lambda_2)$ and $V(2\lambda_2)$ are 
$$S^2({\bf O}_0)=\begin{array}{c}
L(0) \\
| \\
L(2\lambda_2) \\
| \\
L(0)
\end{array}, H^0(2\lambda_2)=\begin{array}{c}
L(0) \\
| \\
L(2\lambda_2) \\
\end{array} \text{ and }  V(2\lambda_2)=\begin{array}{c}
L(2\lambda_2) \\
| \\
L(0)
\end{array},$$ respectively. 
\end{pr}
\begin{proof}
Since $V(2\lambda_2)$ is generated by ${\bf v}_3$, the same arguments as in Proposition \ref{exterioraretilting} show that $V(2\lambda_2)$ contains all ${\bf u}_i{\bf u}_j$, ${\bf v}_i{\bf v}_j$, ${\bf u}_k{\bf v}_l$ with $k\neq l$, and
all $e{\bf u}_i$, $e{\bf v}_i$. Moreover, $e^2-2{\bf u}_i{\bf v}_i\in V(2\lambda_2)$ for every $i$.
Thus the generator of $H^0(0)$ obviously belongs to $V(2\lambda_2)$.  

\end{proof}
Denote $F[{\bf O}_0^n]^G$ by $R(n)$. The algebra $R(n)$ is $\bf N$-graded, say $R(n)=\oplus_{k\geq 0} R(n)_k$. Denote its 
Hilbert-Poincare series $\sum_{i\geq 0}t^k\dim R(n)_k$ by $H_n(t)$.
\begin{pr}\label{hilbert}
If $char F\neq 2$, then $H_n(t)$ does not depend on $char F$. 
\end{pr}
\begin{proof}The algebra $F[{\bf O}_0^n]$ is isomorphic (as an algebraic $G$-algebra) to the symmetric algebra $S(E\otimes {\bf O}_0^*)$, where $E$ is a $n$-dimensional vector space and $G$ acts on the right tensor multiplier ${\bf O}^*_0$. The isomorphism is given by $e_i\otimes a_j\mapsto x_{ij}$, where
the vectors $e_1, \ldots , e_n$ form a basis of $E$ and the vectors $a_1, \ldots , a_7$ form a basis of ${\bf O}_0$. 

A homogeneous component $S^k(E\otimes {\bf O}_0^*)$ has a $GL(E)\times G$-module filtration 
with factors isomorphic to $L_{\lambda}(E)\otimes L_{\lambda}({\bf O}_0^*)$ (cf. \cite{abw}, III). We will call such a filtration an {\it ABW filtration}. A $GL(V)$-module $L_{\lambda}(V)$ is called {\it Schur module} corresponding to a partition $\lambda$ of $k$. In fact, $L_{\lambda}(V)\simeq H^0(GL(V)/B^-(V), F_{\lambda'})$, where $B^-(V)$ is a Borel subgroup of $GL(V)$ consisting of all lower triangular matrices and $\lambda'$ conjugated to $\lambda$.    
Every $L_{\lambda}(V)$ has a finite resolution (as a $GL(V)$-module) with members direct sums of tensor products of various exterior powers
$\Lambda^s(V)$ (see \cite{abw1}, or \cite{don3} for a more general setting). Proposition \ref{exterioraretilting} and Proposition \ref{propofgoodandWeyl} (1) imply that every $S^k(E\otimes {\bf O}_0^*)$ has a good filtration. Since the formal character of $S^k(E\otimes {\bf O}_0^*)$ does not depend on $char F$, Proposition \ref{propofgoodandWeyl} (4) implies that $\dim R(n)_k= c_0(S^k(E\otimes {\bf O}_0^*))$ does not depend on $char F\neq 2$ as well.\end{proof}  
\begin{rem}\label{component} One can prove even more subtle statement that the dimension of each component $R(n)_{k_1, \ldots , k_n}$ does not depend on $char F\neq 2$. In fact, the $G$-module $F[{\bf O}_0^n]_{k_1, \ldots , k_n}\simeq S^{k_1}({\bf O}^*_0)\otimes\ldots\otimes S^{k_n}({\bf O}^*_0)$ has a good filtration and its formal character does not depend on $char F\neq 2$.
\end{rem}

\section{Parameters}

At the beginning of this section assume that $F$ is a perfect field of arbitrary characteristic.

Let $H$ be a reductive group and $V$ be a finite-dimensional $H$-module. The nullcone $\mathcal{N}^H_V$is the zero set of all homogeneous polynomials from $F[V]^H$ that have positive degree. It can be also defined as a closed subvariety consisting of all {\it unstable} points. Recall that a point $v\in V$ is called unstable if and only if $\overline{Hv}$ contains the zero point $0$. By Hilbert-Mumford criterion (cf. \cite{kempf}, Corollary 4.3) a non-zero point $v\in V$ is unstable if and only if there is an one-parameter subgroup $\lambda : G_m\to H$ such that $\lim\limits_{t\rightarrow 0} \lambda(t)v=0$, i.e. $\lambda(t)v=\sum_{i\in I(v)} t^i v_i$ for every $t\in F^*$, where the finite set $I(v)$ consists of positive integers only. Let $Z_{\lambda}$ denote $\{v\in V|  \lim\limits_{t\rightarrow 0} \lambda(t)v=0\}$.

If $I(v)$ is contained in the set of non-negative integers and $v_0\neq 0$, then $v_0$ belongs to 
$\overline{Hv}$. Following \cite{kempf}, we denote
$\lim\limits_{t\rightarrow 0} \lambda(t)v=v_0$. 

A one-parameter subgroup $\lambda : G_m\to G$ is uniquely defined by its action on the vectors ${\bf u}_i$ for $1\leq i\leq 3$. More precisely, 
$\lambda(t)e=e, \lambda(t){\bf u}_i=t^{\lambda_i}{\bf u}_i$ and $\lambda(t){\bf v}_i=t^{-\lambda_i}{\bf v}_i$ for $\sum_{1\leq i\leq 3}\lambda_i =0$. One can identify the set of one-parameter subgroups $\mathbb{Y}(G)$ with a hyperplane in $\mathbb{Z}^3$ via $\lambda\mapsto (\lambda_1, \lambda_2, \lambda_3)$.
  
Let $\mathcal{V}_n$ be the zero set of the polynomials $t(x_i x_j)$ and $t((x_i x_j)x_k)$ for $1\leq i, j, k\leq n$, provided $char F\neq 2$. 
If $char F=2$, then replace $t(x_i^2)$ in the defintion of $\mathcal{V}_n$ by $n(x_i)$ for each $1\leq i\leq n$. 
\begin{pr}\label{nullcone}
If $n\geq 3$, then the set $\mathcal{V}_n$ coincides with the nullcone $\mathcal{N}^G_{{\bf O}^n_0}$.
\end{pr}
\begin{proof}Let ${\bf a}=(a_1,\ldots , a_n)\in {\bf O}_0^n$ belongs to $\mathcal{V}_n$.  If $a_1\neq 0$, then by Lemma \ref{zeronorm} one can assume that $a_1={\bf u}_1$. Let $a_i=\alpha^{(i)} e +{\bf u}^{(i)} + {\bf v}^{(i)}$ for $i\geq 2$. An equation $t({\bf u}_1 a_i)=0$ implies that ${\bf u}_1\cdot {\bf v}^{(i)}=0$, hence every ${\bf v}^{(i)}$ belongs to the subspace $F{\bf v}_2 + F{\bf v}_3$. The equation $t(({\bf u}_1 a_i)a_j)=0$ implies
that $({\bf u}_1{\bf u}^{(i)}\cdot {\bf u}^{(j)})=0$.
Thus, for arbitrary $i\geq 2$ and $j\geq 2$, the subspace generated by ${\bf u}_1, {\bf u}^{(i)}$ and ${\bf u}^{(j)}$ is at most two-dimensional. 

If all ${\bf u}^{(i)}$ belong to $F{\bf u}_1$, then
all $\alpha^{(i)}$ are equal to zero, hence $\lim\limits_{t\rightarrow 0}\lambda_{2, -1, -1}(t){\bf a}=0$.
Otherwise,  there is ${\bf u}^{(i)}$ such that $V=F{\bf u}_1 +F{\bf u}^{(i)}$ is two-dimensional, hence $V$ contains all other ${\bf u}^{(j)}$. 

Without a loss of generality one can assume that $i=2$. There is an element $g\in SL_3(F)$ such that ${\bf u}_1 g={\bf u}_1$ and ${\bf u}^{(2)}g={\bf u}_2$. From now on 
we assume ${\bf u}^{(2)}={\bf u}_2$.

We have $\lim\limits_{t\rightarrow 0}\lambda_{1, 0, -1}(t){\bf a}=(0, b_2,\ldots , b_n)\in \overline{G{\bf a}}$, where
$b_j=\alpha^{(j)}e +\beta_j {\bf u}_2 +\gamma_j {\bf v}_2$ for $j\geq 2$, and $\beta_2=1$. Since ${\bf b}=(b_2,\ldots , b_n)$ belongs to $\mathcal{V}_{n-1}$, either one concludes the proof by induction on $n$ or one has $n=3$. 

Now assume that $n=3$. Since $n(b_2)=0$, we obtain $X_{-\omega_2}(-\alpha^{(2)})b_2=
{\bf u}_2$. The element $X_{-\omega_2}(-\alpha^{(2)})b_3$ belongs to the span of $e, {\bf u}_2, {\bf v}_2$. 
Since $0=tr({\bf u}_2 b_3)=\gamma_3$, we derive $\alpha^{(3)}=0$. Now it is clear that $\lim\limits_{t\rightarrow 0}\lambda_{0, 1, -1}(t){\bf b}=0$.\end{proof}

From now on, till the end of this section, we assume that $char F\neq 2$. Recall that 
$G\leq SO(q|_{{\bf O}_0})=SO(7)$. Here $q(a, b)=-t(ab)$ for every $a, b\in {\bf O}_0$. As it has been proved in \cite{dp}, the algebra
$F[{\bf O}_0^n]^{O(7)}$ is generated by the polynomials $t(x_i x_j)$ for $1\leq i, j\leq n$. The space $({\bf O}_0, q|_{{\bf O}_0})$ has an orthogonal decomposition
$Fe\ \bot \ (F{\bf u}_1+F{\bf v}_1) \ \bot \ (F{\bf u}_2+F{\bf v}_2) \ \bot \ (F{\bf u}_3+F{\bf v}_3)$. Denote $(F{\bf u}_1+F{\bf v}_1) \ \bot \ (F{\bf u}_2+F{\bf v}_2)$ by $W$. The natural isometric embedding
$W\to {\bf O}_0$ induces an epimorphism $F[{\bf O}_0^n]^{O(7)}\to F[W^n]^{O(4)}$, defined by $t(x_i x_j)\mapsto t(x'_i x'_j)$, where $x'_i=
x_i|_{W}$ for $1\leq i\leq n$. 
\begin{pr}\label{Krulldim}
If $n\geq 3$, then $\dim R(n)=7n-14$.
\end{pr}  
\begin{proof}Let $\pi : {\bf O}_0^n\to {\bf O}_0^n //G$ be a quotient morphism that is dual to the algebra embedding $F[{\bf O}_0^n]^{G}\to F[{\bf O}_0^n]$. For the generic elements $x_1, x_2, x_3$ define  
$$y_i=x_i \text{ for } 1\leq i\leq 3, y_4=x_1 x_2-\frac{t(x_1 x_2)}{2}, y_5=x_1 x_3-\frac{t(x_1 x_3)}{2},$$$$y_6=x_2 x_3-\frac{t(x_2 x_3)}{2}, \text{ and} \ 
y_7=(x_1 x_2)x_3-\frac{t((x_1 x_2)x_3)}{2}.$$
Each $y_i$ has the form $y_{i1}e +\sum_{1\leq j\leq 3}y_{i, j+1}{\bf u}_i +\sum_{1\leq j\leq 3}y_{i, j+4 }{\bf v}_i$.
Set $Y=(y_{ij})_{1\leq i, j\leq 7}$ and denote $\det(Y)$ by $f=f(x_1, x_2, x_3)$. It is clear that $f\in R(3)$.  

Observe that $f({\bf u}_1, {\bf u}_2, {\bf u}_3)\neq 0$.
In particular, the open subset $\mathcal{O}=({\bf O}_0^n)_f$ is a non-empty $G$-subvariety in ${\bf O}_0^n$. By the standard properties of
categorical quotients (see \cite{nw} for all necessary details), $\pi|_{\mathcal{O}}$ induces a quotient map $\mathcal{O}\to \mathcal{O}//G=({\bf O}_0^n//G)_f$.
Since the stabilizer (in $G$) of any point from $\mathcal{O}$ is trivial, $\mathcal{O}//G$ is a geometric quotient (cf. \cite{dk}, pp.53-54).
Theorem 4.3 of \cite{ham} implies $$\dim {\bf O}_0^n//G =\dim \mathcal{O}//G =\dim \mathcal{O}-\dim G=\dim {\bf O}_0^n-\dim G =7n-14 .$$\end{proof}  
\begin{lm}\label{Krulldimfororthogonal}
If $n\geq 3$, then the Krull dimension of $F[W^n]^{O(4)}$ is equal to $4n-6$.  
\end{lm}
\begin{proof}
Let $d$ denote the determinant of Gram matrix $(t(x'_i x'_j))_{1\leq i, j\leq 3}$. 
It is clear that $d$ belongs to $F[W]^{O(4)}$. If ${\bf w}=(w_1, w_2, w_3, \ldots)\in\mathcal{W}=(W^n)_d$, then the elements $w_1, w_2, w_3$ are linearly independent and generate a subspace $V\subseteq W$ on which $q$ is non-degenerate. Thus $W=V\perp Fw$ and the stabilizer of ${\bf w}$ consists of at most two elements, where one of them might take $w$ to $-w$. Therefore, all orbits in $\mathcal{W}$ have dimension $6$  and our statement follows as in Proposition \ref{Krulldim}. 
\end{proof}
\begin{pr}\label{codimoforthogonal}
If $n\geq 3$, then the nullcone $\mathcal{N}^{O(7)}_{{\bf O}_0^n}$ has codimension $4n-6$ in ${\bf O}_0^n$. 
\end{pr}
\begin{proof}A maximal torus $T\simeq G_m^3$ in $H=O(7)$ can be chosen so that each element $t\in T$ acts on ${\bf O}_0$ by $te=e, t{\bf u}_i=t_i {\bf u}_i$ 
and $t{\bf v}_i=t_i^{-1}{\bf v}_i$ for $1\leq i\leq 3$.  

By Hilbert-Mamford criterion, $\mathcal{N}^H_{{\bf O}_0^n}=\bigcup_{\lambda} HZ_{\lambda}$. If $\lambda=\lambda_{k_1, k_2, k_3}$, then
$$Z_{\lambda}=Z_{k_1, k_2, k_3}=\{{\bf a}\in {\bf O}_0^n | a_i\in \sum_{k_j >0} F{\bf u}_j +\sum_{k_j <0} F{\bf v}_j \text{ for } 1\leq i\leq n\}$$
is a subspace in ${\bf O}_0^n$. It is clear that $Z_{k_1, k_2, k_3}=Z_{t_1, t_2, t_3}$, where $t_i=\frac {k_i}{|k_i|}$ if $k_i\neq 0$, and
$t_i=k_i$ otherwise. 

Every $Z_{\lambda}$ is a subspace of $V^n$, where $V$ is a maximal totally isotropic subspace in ${\bf O}_0$. Moreover, if all coordinates of $\lambda$ are non-zero, then $Z_{\lambda}=V^n$.  Since all such subspaces $V$ are conjugated by isometries from $H$, it follows that $\mathcal{N}^H_{{\bf O}_0^n}=HZ_{1, 1, 1}=\overline{HZ_{1, 1, 1}}$. 

Denote $Z_{1, 1, 1}$ by $Z$ and
$\sum_{1\leq i\leq 3}F {\bf u}_i$ by $V$. 
The subgroup $SO(7)=H^{\circ}$ has index two in $H$. Thus $HZ$ has at most two irreducible components of the same dimension.
In particular, $\dim \mathcal{N}^H_{{\bf O}_0^n}=\dim \overline{H^0Z}
=\dim H^0Z$.

Consider a natural dominant morphism of irreducible affine varieties $\alpha : H^{\circ}\times Z\to \overline{H^0Z}$.      
This morphism can be factored through $H^{\circ}\times_P Z=(H^{\circ}\times Z)//P$, where $P=Stab_{H^{\circ}}(V)$ acts on
$H^{\circ}\times Z$ by the rule $(h, z)p=(hp, p^{-1}z)$ for $h\in H^{\circ}, p\in P$ and $z\in Z$. Since the stabilizer of any point 
from $H^{\circ}\times Z$ in $P$ is trivial, 
$H^{\circ}\times_P Z$ is also a geometric quotient. In particular,
$\dim H^{\circ}Z\leq \dim (H^{\circ}\times_P Z)=\dim H^{\circ} +\dim Z- \dim P=3n +6$.    

By theorem 5.2 from \cite{ham}, it remains to find an open subset $U\subseteq H^{\circ}\times Z$ such that
for every $u\in U$ the differential $d_u\alpha$ has rank at least $3n+6$. Set $U=H^{\circ}\times W$, where
$W=\{{\bf a}\in Z | a_1\wedge a_2\wedge a_3\neq 0\}$. Since $U$ is $H^{\circ}$-subvariety, all we need is to compute the rank of $d_{(1, {\bf a})}\alpha$, where ${\bf a}\in W$.
We have $d_{(1, {\bf a})}\alpha (x, {\bf b})=x{\bf a} + {\bf b}$ for $x\in Lie(H^{\circ})$ and ${\bf b}\in Z$. Thus 
$$\ker d_{(1, {\bf a})}\alpha=\{(x, -x{\bf a})| x\in Lie(H^{\circ})\}\bigcap (Lie(H^{\circ})\times Z) \simeq Lie(P),$$
i.e. the rank of $d_{(1, {\bf a})}\alpha$ equals $3n+6$. The proposition is proved.\end{proof}  
\begin{theorem}\label{parameters}
If $n\geq 3$, then $R(n)$ has a system of parameters consisting of $4n-6$ homogeneous invariants of degree 2 and 
$3n-8$ homogeneous invariants of degree 3.  
\end{theorem}
\begin{proof}The arguments from \cite{schw2}, Lemma 7.11 does not depend on $char F$. Therefore, one can choose $k=4n-6$ linear combinations
$h_1, \ldots , h_k$ of the elements $t(x_i x_j)$ such that their images in $F[W^n]^{O(4)}$ form a system of parameters and 
the zero set of $h_1, \ldots , h_k$ has a codimension $k$ in ${\bf O}_0^n$. Thus $h_1, \ldots , h_k$ is a part of a system of parameters
in $R(n)$. 
Further, the relation $Rel_8$ from \cite{schw2} can be reduced modulo
$p\neq 2$. Again, arguing as in \cite{schw2}, Theorem 7.13,  we obtain that the partial system of parameters $h_1, \ldots, h_k$ can be extended to a complete one
by $3n-8$ linear combinations of $t((x_i x_j) x_k)$.\end{proof}   
\begin{cor}\label{R(3)}
The algebra $R(3)$ is freely generated by the elements $t(x_ix_j)$ for $1\leq i\leq j\leq 3$ and by $t((x_1 x_2)x_3)$. 
\end{cor}
\begin{proof}By Theorem \ref{parameters}, these elements are parameters. Proposition 3.2 combined with \cite{schw2}, Remark 7.14, imply our statement.\end{proof} 
\begin{theorem}\label{rationalinv}
The field of rational $G$-invariants is generated by the elements 
$$t(x_i x_j) \text{ for } 1\leq i\leq j\leq 3, t((x_1 x_2) x_3); t(x_i x_k) \text{ for } 1\leq i\leq 3, 4\leq k\leq n;$$
$$t((x_i x_j) x_k) \text{ for } 1\leq i < j\leq 3, 4\leq k\leq n; \text{ and } t(((x_1 x_2)x_3)x_k) \text{ for } 4\leq k\leq n.$$
\end{theorem}
\begin{proof}Since $G$ is connected, it is enough to show that any homogeneous polynomial $G$-invariant $f$ can be represented as a rational function
in these traces. Denote by $\Delta$ the determinant of Gram matrix $(t(y_i y_j))_{1\leq i, j\leq 7}$, where $y_i$ are the elements from Proposition \ref{Krulldim}. 
Consider the generic (traceless) elements $z_{ij}$ for $1\leq i\leq n$ and $1\leq j\leq 7$. 
The polynomial $f(\sum_{1\leq j\leq 7} z_{1j}, \ldots , \sum_{1\leq j\leq 7} z_{nj})$ and all its homogeneous components
$f_{\alpha}$ are obviously $G$-invariant, where $\alpha=(\alpha_{ij})_{1\leq i\leq n, 1\leq j\leq 7}$ and $f_{\alpha}$ has degree $\alpha_{ij}$
in each $z_{ij}$. 

By linear algebra arguments, there are $c_{ij}\in R(3)_{\Delta}[t(x_i y_j) |1\leq i\leq n, 1\leq j\leq 7]$ such that $x_i=\sum_{1\leq j\leq 7}c_{ij}y_j$ for $1\leq i\leq n$.
Then  $$f=f(x_1, \ldots , x_n)=f(\sum_{1\leq j\leq 7} c_{1j}y_i, \ldots, \sum_{1\leq j\leq 7} c_{ni}y_j)=
\sum_{\alpha} (\prod_{1\leq i\leq n, 1\leq j\leq 7}c_{ij}^{\alpha_{ij}}) f'_{\alpha}.$$
The observation that every $f'_{\alpha}$ is obtained from $f_{\alpha}$ by 
specializing $z_{ij}\to y_j$ concludes the proof.\end{proof}        
\begin{cor}\label{rationality}
The field of rational invariants is a field of rational functions over $F$ in $7n-14$ variables.
\end{cor}
\begin{proof}
By Theorem 2.3 from \cite{shestpolikarp}, the generators from Theorem \ref{rationalinv} are algebraically independent.
\end{proof}

\section{Quotients $SO(7)/G$ and $Spin(7)/G$}

Throughout this section, $F$ is an algebraically closed field and $char F\neq 2$.

Denote $\{a\in {\bf O}| n(a)=1\}$ by ${\bf O}_1$ and ${\bf O}_1/\{\pm 1_{\bf O}\}$ by $M$.
Then ${\bf O}_1$ and $M$ are (non-associative) {\it Moufang loops} \cite{do, nv}.
According to Theorem 4.6 of \cite{nv}, the group $SO(8)=SO(q)$ is generated by the operators $R_x : z\mapsto zx$ and  $L_x : z\mapsto xz$, where $x\in {\bf O}_1$ and $z\in {\bf O}$.
In other words, the group $SO(8)$ can be identified with the {\it multiplication group} $Mult({\bf O}_1)$ of ${\bf O}_1$, generated by all left and right translations. 
Contrary to the standard definitions in the theory of Moufang loops and groups with triality
we assume that all actions are left actions.

Furthermore, the factor-group $PSO(8)=SO(8)/\{\pm E\}$ is identified with $Mult(M)$. According to \cite{glaub}, since the {\it nucleus} $N(M)$ is trivial,
the group $PSO(8)$ is a group with triality in the sense of Definition 5.1 of
\cite{nv} (see also \cite{do}). More precisely, denote $P_x =L_x^{-1}R_x^{-1}$ for $x\in M$. The triality automorphisms $\sigma$ and $\rho$ are defined on 
the generators $R_x$ and $L_x$ by 
$$R_x^{\rho}=P_x, L_x^{\rho}=R_x, P_x^{\rho}=L_x, \text{ and}$$
$$P_x^{\sigma}=P_x^{-1}, R_x^{\sigma}=L_x^{-1}, L_x^{\sigma}=R_x^{-1}.$$ 
The automorphisms $\sigma$ and $\rho$ satisfy the relations
$$
\sigma^2=\rho^3=1, \sigma\rho\sigma=\rho^{-1},
$$  
and therefore generate a dihedral group $D$ of order $6$.

Since we have replaced actions on the right by actions on the left, the defining relations of $PSO(8)$ are expressed as follows (compare with \cite{do}, p.383):
$$U_x U_y U_x=U_{xyx} \text{ for } U\in \{P, L, R\}, \text{ and }$$
$$L_y P_x R_y=P_{xy^{-1}}, R_y P_x L_y=P_{y^{-1}x} \text{ for } x, y\in M,$$
and all the remaining relations are obtained by $\rho$-shift. 

By the remark in the proof of Theorem 7.4 of \cite{nv}, the triality automorphisms are exactly the graph automorphisms of $PSO(8)$. 
Define the {\it inner mapping group} $Inn(M)=\{\phi\in Mult(M) |\phi(1_{\bf O})=1_{\bf O}\}$. This subgroup is generated by the elements
$T_x=R_x L_x^{-1}, R_{x, y}=R^{-1}_{xy} R_y R_x$ and $L_{x, y}=L_{yx}^{-1} L_y L_x$ for $x, y\in M$.
As above, $SO(7)=Stab_{SO(8)}(1_{\bf O})$, and since the central symmetry $z\mapsto -z$ does not belong to $SO(7)$, $SO(7)$ is embedded into $PSO(8)$
and identified with $Inn(M)$.
Thus $SO(7)=\{g\in PSO(8)| g^\sigma=g\}$ and $G=\{g\in SO(7)| g^{\rho}=g\}$ (cf. \cite{nv}, p.26). For $g, h\in PSO(8)$ we denote
$hgh^{-1}$ by $g^h$.

\begin{pr}\label{epi}
The map $f : h\mapsto h^{\rho^2}h^{-\rho}$ is a surjection from $SO(7)$ onto $V=\{P_x |x\in M\}$, and its fibers are exactly
left cosets of $G$ in $SO(7)$. 
\end{pr}
\begin{proof}As in Theorem 2 of \cite{do}, we work in the semidirect product $D\ltimes G$. The relations
$$(\sigma P_x)^{L_y}=\sigma P_{yx} , (\sigma P_x)^{R_y}=\sigma P_{xy},
(\sigma P_x)^{P_y}=\sigma P_{y^{-1}xy^{-1}}$$ imply that $V=\{[\sigma, x]=x^{-\sigma}x|
x\in PSO(8)\}$. Since $h^{\rho^2}h^{-\rho}=[\sigma, h^{-\rho}]$ for every $h\in SO(7)$, we infer that $f$ maps $SO(7)$ to $V$.    
The statement about fibers is now obvious. For every $h, g\in SO(7)$ we have $f(gh)=g^{\rho^2}f(h) g^{-\rho}$. In particular, if $g\in G$, then $f(gh)=f(h^g)=f(h)^g$. 

Observe that $L_x^g=L_{gx}, R_x^g=R_{gx}$ for every $x\in {\bf O}_1$ and $g\in SO(8)$. Analogous relations hold for all operators $P_x, T_x, L_{x, y}$ and $R_{x, y}$.

The equation $g^{\rho^2} P_z g^{-\rho}= \sigma (g^{\rho}\sigma P_z g^{-\rho})$ 
and the above relations applied to the generators $T_x, R_{x, y}$ and $L_{x, y}$ imply
$$f(T_x h)=P_{x^{-1}z x^{-2}}, f(L_{x, y}h)=P_{((zx)y)(yx)^{-1}} \ \mbox{and}$$
$$f(R_{x, y}h)=P_{(xy)(y^{-1}(x^{-1}zx^{-1})y^{-1})(xy)},$$ where $f(h)=P_z$ and $z\in M$. 

Since $f(T_{x^{-1}})=P_{x^3}$, the statement will follow if for every $z\in M$ one finds 
an octonion $x\in {\bf O}_1$ such that $x^3=z$.
Let $z=\alpha e_1 +\beta e_2 +{\bf u}+{\bf v}$ and $x=\alpha' e_1 +\beta' e_2 +{\bf u}'+{\bf v}'$. The quadratic equation $x^2-t(x)x+1=0$ implies $x^3=(t(x)^2 -1)x-t(x)$. Therefore, 
$x^3=z$ if and only if $(t(x)^2 -1)x=z+t(x)$. Thus $t(x)$ is a root of the equation $t^3-3t-t(z)=0$. All roots of this equation are $\pm 1$ if and only if $char F=3$ and $t(z)=\pm 1$.

Assume that either $char F\neq 3$ or $t(z)\neq\pm 1$. Then there is a root of this equation, say $t_1$,  that is not equal to $\pm 1$.
Set
$$\alpha'=\frac{\alpha + t_1}{t_1^2 -1}, \beta'=\frac{\beta + t_1}{t_1^2 -1}, {\bf u}'=
\frac{{\bf u}}{t_1^2 -1} \text{ and } {\bf v}'=\frac{{\bf v}}{t_1^2 -1}.$$    
It is easy to check that $t(x)=t_1$ and $x\in {\bf O}_1$, hence $x^3=z$. 

Finally, let $char F=3$ and $t(z)=\pm 1$. Since $z$ and $-z$ represent the same element in $M$, we can assume that $t(z)=1$. Thus $(z+1)^2=0$ and $t(z+1)=0$ imply $n(z+1)=0$. By Lemma \ref{zeronorm}, one can assume that $z=-1+{\bf u}_1$.

Set $x={\bf u}_2-{\bf v}_2$ and $y=e_2+x$. Elementary computations show that $t(z')=0$, where
 $z'=((zx)y)(yx)^{-1}$. 
As it has been shown above, there is $x'$ such that $(x')^3=z'$, hence $P_z=f(L_{x, y}^{-1}T_{x'})$, which concludes the proof of this proposition.\end{proof}
\begin{lm}\label{affinevar}
The subset $V$ is a closed irreducible subvariety of $PSO(8)$. In particular, $f$ is a surjective  morphism of affine varieties. 
\end{lm}
\begin{proof}
Let $W$ denote the subset $\{R_z | z\in {\bf O}_1\}\subseteq SO(8)$. It is clear that $W$ is a closed (affine)
subvariety of $SO(8)$. In fact, if $g\in SO(8)$ has a matrix $(g_{ij})_{1\leq i, j\leq 8}$ with respect to the basis
$e_1, e_2, {\bf u}_i, {\bf v}_i$ for $1\leq i\leq 3$, then $g=R_z$ if and only if 
$$z={\bf g}_1 + {\bf g}_2 \ \mbox{and} \ {\bf u}_i ({\bf g}_1 + {\bf g}_2)={\bf g}_{i+2}, {\bf v}_i ({\bf g}_1 + {\bf g}_2)={\bf g}_{i+5} \text{ for } 1\leq i\leq 3,$$ where ${\bf g}_i=g_{i 1}e_1 +g_{i 2}e_2 +\sum_{3\leq k\leq 5} g_{ik}{\bf u}_{k-2} +\sum_{6\leq k\leq 8}g_{ik}{\bf v}_{k-5}$. The map $z\mapsto R_z$ induces an isomorphism (of affine varieties) ${\bf O}_1\simeq W$. Since ${\bf O}_1$ is irreducible, so is $W$. Furthermore, $V'=W/\{\pm E\}$ is a closed irreducible subvariety of $PSO(8)$. The claim follows from $V'^{\rho}=V$.
\end{proof}
\begin{lm}\label{separable}
The induced morphism $W\to V$ is separable.
\end{lm}
\begin{proof}
It is enough to show that $W\to V'$ is separable. The algebra $F[W]$ has a $Z_2$-grading $F[W]=F[W]_0\oplus F[W]_1$, where $F[W]^{\{\pm E\}}=F[W]_0=F[V']$ and $F[W]_1=\{h\in F[W]| h(-w)=-h(w) \ \mbox{for every} \ w\in W\}$. Consider a rational function $\frac{u}{v}\in
Q(W)$, where $u, v\in F[W]$. Then $\frac{u}{v}$ is a root of the separable polynomial $$t^2-\frac{u_0 v_0 -u_1 v_1}{v_0^2 -v_1^2}t +\frac{u_0^2-u_1^2}{v_0^2-v_1^2}$$ 
and its coefficients belong to $Q(V')$. 
\end{proof}
Denote by $K$ the subgroup of $SO(8)$ generated by the operators $T'_x : z\mapsto x^{-1} zx^{-2}$, $L'_{x, y} : z\mapsto
((zx)y)(yx)^{-1}$ and $R'_{x, y} : z\mapsto  (xy)(y^{-1}(x^{-1}zx^{-1})y^{-1})(xy)$ for $x, y\in {\bf O}_1$ and $z\in {\bf O}$. By Proposition 7.5 of \cite{ham}, $K$ is closed and connected.
\begin{rem}\label{aboutK}
If we consider $T'_x$, $R'_{x, y}$ and $L'_{x, y}$ as elements of $PSO(8)$, then they coincide with $T_x^{\rho}$, $R^{\rho}_{x, y}$
and $L_{x, y}^{\rho}$ respectively. In particular, there is an epimorphism $K\to SO(7)^{\rho}$ and its  kernel is equal to
$\{\pm E\}=\{E, -E=T'_{-1_{\bf O}}\}$. Since $F$ is algebraically closed, $K\simeq Spin(7)$. Also, for every $g, h\in SO(7)$ we have $f(gh)=
P_{g^{\rho}z}$, where $P_z=f(h)$.
\end{rem}

The differential $d\rho$ is an automorphism of Lie algebra $L=Lie(PSO(8))=Lie(SO(8))$. Thus $Lie(K)=d\rho(Lie(SO(7))$.

The operators $E_{ij}\in End_F({\bf O})$ are defined by $E_{ij}e_k=\delta_{jk}e_i$ for $1\leq i, j, k\leq 8$. Choose a maximal torus of $SO(8)$ as      
$$T=\{t_1 E_{11}+t_1^{-1}E_{22}+\sum_{3\leq i\leq 5}(t_{i-1} E_{ii}+t_{i-1}^{-1}E_{i+3, i+3})| t_i\in F^*\}.$$
The algebra $L$ has a basis
$$E_{11}-E_{22}, \ E_{i 1}+E_{2, i+3} \text{ for } 3\leq i\leq 5; \ E_{i 1} +E_{2, i-3} \text{ for } 6\leq i\leq 8;$$ 
$$E_{i 2}+E_{1, i+3} \text{ for } 3\leq i\leq 5; \ E_{i 2}+E_{1, i-3} \text{ for } 6\leq i\leq 8;$$
$$E_{ij}-E_{j+3, i+3} \text{ for } 3\leq i, j\leq 5;$$
$$E_{ji} - E_{i+3, j-3} \text{ for }3\leq i\leq 5, i+3 < j\leq 8;$$
$$E_{ji} -E_{i-3, j+3} \text{ for } 6\leq i\leq 8, 3\leq j <i-3,$$
consisting of weight vectors with respect to the adjoint action of $T$.
The subalgebra $Lie(T)$ is generated by the elements 
$$H_1=E_{11}-E_{22} \ \mbox{and} \ H_{i-1}=E_{ii}-E_{i+3, i+3} \text{ for } 3\leq i\leq 5.$$
The remaining basic elements have the weights 
$$-\omega_1 +\omega_{i-1} \text{ for } 3\leq i\leq 5; \ -\omega_1 -\omega_{i-4} \text{ for } 6\leq i\leq 8;$$
$$\omega_1 +\omega_{i-1} \text{ for } 3\leq i\leq 5; \ \omega_1 -\omega_{i-4} \text{ for } 6\leq i\leq 8;$$
$$\omega_{i-1} -\omega_{j-1} \text{ for } 3\leq i\neq j\leq 5;$$
$$-\omega_{i-1} -\omega_{j-4} \text{ for } 3\leq i\leq 5, i+3 < j\leq 8;$$
$$\omega_{i-4} +\omega_{j-1} \text{ for } 6\leq i\leq 8, 3\leq j <i-3$$
respectively. For each weight $\alpha\neq 0$ from this list, the related basic element is denoted by $Y_{\alpha}$ and the corresponding root subgroup is defined by 
$X_{\alpha}(t)=E+tY_{\alpha}$ for $t\in F$. 

Choose a set of positive roots as $\{\omega_i \pm \omega_j| 1\leq i < j\leq 4\}$. Its subset of simple roots is $\{\omega_i-\omega_{i+1}, \omega_3+\omega_4| 1\leq i\leq 3\}$ and  
the fundamental dominant weights are $$\lambda_1=\omega_1, 
\lambda_2=\omega_1+\omega_2, \lambda_3=\frac{1}{2}(\omega_1+\omega_2+\omega_3-\omega_4) \text{ and } \lambda_4=\frac{1}{2}(\omega_1+\omega_2+\omega_3+\omega_4).$$

\begin{lm}\label{SO(7)issaturated}
The subgroup $SO(7)$ is saturated in $SO(8)$ and $PSO(8)$.
\end{lm}
\begin{proof}There is an orbit morphism $\pi : SO(8)\to {\bf O}_1$ given by $g\mapsto g1_{\bf O}$ for $g\in SO(8)$. The tangent space $T_{1_{\bf O}}({\bf O}_1)$ can be identified with the subspace ${\bf O}_0$ of $T_{1_{\bf O}}({\bf O})={\bf O}$.

The differential $d_{E}\pi$ maps each $X\in Lie(SO(8))$ to $X 1_{\bf O}$. 
It is easy to see that $d_{E}\pi$ is surjective. By \cite{ham}, 12.4, $\pi$ is identified with the factor-morphism $SO(8)\to SO(8)/SO(7)$. 
By Proposition \ref{goodsubgroup}, one has to prove that the left $SO(8)$-module $F[{\bf O}_1]$ has a good filtration. 

We have $F[{\bf O}_1]\simeq
F[{\bf O}]/F[{\bf O}](n(z)-1)$, where ${\bf O}$ is regarded as the standard left $SO(8)$-module and $z$ is a generic octonion related to ${\bf O}$. Thus $F[{\bf O}](n(z)-1)\simeq F[{\bf O}]$ as a $SO(8)$-module and by Proposition \ref{propofgoodandWeyl} (1), it is enough to verify that each
component $F[{\bf O}]_k\simeq S^k ({\bf O}^*)\simeq S^k({\bf O})$ has a good filtration. 

For every vector space $U$ there is an (acyclic) Koszul complex (of $GL(U)$-modules, see \cite{abw}, Definition V.1.3 and Corollary V.1.15)
$$0\to \Lambda^k(U)\to\ldots\to\Lambda^i(U)\otimes S^{k-i}(U)\to\ldots\to S^k(U)\to 0.$$  
Since all exterior powers $\Lambda^i({\bf O})$ are tilting $SO(8)$-modules, the induction on
$k$ implies the statement for $SO(8)$. For the case of the group $PSO(8)$, observe that $F[PSO(8)/SO(7)]\simeq F[{\bf O}_1]^{\{\pm E\}}$ is a direct summand of $F[{\bf O}_1]$ and use Lemma 3.1.3 of \cite{don}, combined with \cite{ham}, p.162, Exercise 2.\end{proof}  

The following lemma is a Lie analogue of Theorem 4.6 of \cite{nv}.
\begin{lm}\label{Liecounterpart}
The algebra $L$ is generated by the operators $l_a : z\to az$ and $r_a : z\to za$ for $a\in {\bf O}_0$, where $z\in {\bf O}$. 
\end{lm}
\begin{proof}Denote by $S$ the subalgebra of $L$ generated by all $l_a$ and $r_a$. The equalities
$$\frac{1}{2}(l_e +r_e)=H_1, \ [H_1, l_{{\bf u}_i}]=Y_{\omega_1 +\omega_{i+1}}, \
[H_1, l_{{\bf v}_i}]=-Y_{-\omega_1 -\omega_{i+1}},$$ 
$$[H_1, r_{{\bf u}_i}]=-Y_{-\omega_1 +\omega_{i+1}}, \
[H_1, r_{{\bf v}_i}]=Y_{\omega_1 -\omega_{i+1}} \text{ for } \ 1\leq i\leq 3,
$$
$$[Y_{\omega_1+\omega_i}, Y_{-\omega_1+\omega_j}]=Y_{\omega_i +\omega_j}, \ [Y_{-\omega_1-\omega_j}, Y_{\omega_1-\omega_i}]=Y_{-\omega_i-\omega_j} \text{ for } 2\leq i<j\leq 4,$$ 
$$
[Y_{-\omega_i-\omega_j}, Y_{\omega_i+\omega_j}]=-H_i-H_j \text{ for } 2\leq i< j\leq 4 \text{ and }
$$
$$[r_{u_i}, l_{v_j}]=-Y_{\omega_{i+1}-\omega_{j+1}} \text{ for } 1\leq i\neq j\leq 3,$$
imply $L\subseteq S$, hence $S=L$.
\end{proof}
The subalgebra $Lie(SO(7))\subseteq L$ has a basis
$$
Y_{-\omega_1 +\omega_i}-Y_{\omega_1 +\omega_i},  \ Y_{-\omega_1 -\omega_i}-Y_{\omega_1 -\omega_i},
\ Y_{\omega_i -\omega_j} \text{ for } 2\leq i\neq j\leq 4;$$
$$Y_{-\omega_i -\omega_j}, \ Y_{\omega_i +\omega_j} \text{ for } 2\leq i< j\leq 4 \text{ and }
 H_i \text{ for } 2\leq i\leq 4.$$
The basic elements $Y_{-\omega_1 +\omega_i}-Y_{\omega_1 +\omega_i}$ and $Y_{-\omega_1 -\omega_i}-Y_{\omega_1 -\omega_i}$ have weights $\omega_i$ and $-\omega_i$ with respect to the adjoint action of the torus 
$$T'=\{E_{11}+E_{22}+\sum_{3\leq i\leq 5}(t_{i-1} E_{ii}+t_{i-1}^{-1}E_{i+3, i+3})| t_i\in F^*\}$$
of the subgroup $SO(7)=Stab_{SO(8)}(1_{\bf O})$. 

The corresponding root system is 
$$\{\pm\omega_i |2\leq i\leq 4\}\cup\{\pm(\omega_i +\omega_j) | 2\leq i < j\leq 4\}\cup\{\pm(\omega_i-\omega_j) |2\leq i <j\leq 4\}.$$ 
Its subset of positive roots is $\{\omega_i |2\leq i\leq 4\}\cup\{\omega_i\pm\omega_j |2\leq i< j\leq 4\}$, the simple roots and fundamental dominant weights are
$$
\omega_2-\omega_3, \omega_3-\omega_4, \omega_4,
$$
and
$$\lambda_1'=\omega_2, \lambda_2'=\omega_2+\omega_3, \lambda_3'=\frac{1}{2}(\omega_2+\omega_3 +\omega_4),$$
respectively. 
\begin{lm}\label{actionofdiff}
The differential $d\rho$ acts on $L$ as follows.
$$d\rho(H_1)=-\frac{1}{2}(\sum_{1\leq i\leq 4}H_i), \ d\rho(H_i)=\frac{1}{2}H_1-\frac{1}{2}(\sum_{j\neq 1, i}H_j-H_i) \text{ for } 2\leq i\leq 4;$$
$$d\rho(Y_{\omega_1+\omega_i})=(-1)^{i+1}Y_{-\omega_j-\omega_k}, \
d\rho(Y_{-\omega_1-\omega_i})=(-1)^{i+1}Y_{\omega_j+\omega_k} \text{ for } 2\leq i\neq j\neq k\leq 4;$$
$$
d\rho(Y_{-\omega_1+\omega_i})=-Y_{\omega_1+\omega_i}, \ d\rho(Y_{\omega_1-\omega_i})=-Y_{-\omega_1-\omega_i} \text{ for } 2\leq i\leq 4;
$$
$$
d\rho(Y_{-\omega_j-\omega_k})=(-1)^i Y_{-\omega_1+\omega_i}, \ d\rho(Y_{\omega_j+\omega_k})=(-1)^i Y_{\omega_1-\omega_i} \text{ for } 2\leq i\neq j\neq k\leq 4 \text{ and }
$$
$$
d\rho(Y_{\omega_i -\omega_j})=Y_{\omega_i -\omega_j} \text{ for } 2\leq i\neq j\leq 4.
$$
\end{lm}
\begin{proof}
Observe that $d\rho$ maps $l_a$ to $r_a$ and $r_a$ to $-(l_a +r_a)$ respectively and use the equalities from Lemma \ref{Liecounterpart}.
\end{proof}
\begin{theorem}\label{affinestructureofV}
The morphism $f : SO(7)\to V$ can be identified with the factor-morphism $SO(7)\to SO(7)/G$.
\end{theorem}
\begin{proof} By 12.4 of \cite{ham}, it is sufficient to prove that $f$ is separable. In notations of Lemma \ref{affinevar}, there is a commutative diagram
$$\begin{array}{ccc}
SO(7) & \stackrel{f}{\to} & V \\
\uparrow & & \uparrow \\
K & \stackrel{h}{\to} & W \\
\end{array},$$
where $h : K\to W$ is a composition of the orbit map $g\mapsto g 1_{\bf O}$ and the isomorphism $z\mapsto R_z$. Correspondingly, $K\to SO(7)$ is a composition of factor-morphism $K\to SO(7)^{\rho}$ and the automorphism $\rho^{-1}=\rho^2$.

The groups $K$ and $SO(7)$ act transitively on $W$ and $V$, respectively. Thus $W$ and $V$ are smooth varieties of the same dimension. Since  $Lie(K)=Lie(SO(7))$, by Theorem 5.5, \cite{ham}, $f$ is separable if and only if $h$ 
is if and only if the orbit map $K\to {\bf O}_1$ is separable.

The tangent space $T_{1_{\bf O}}({\bf O}_1)$ coincides with ${\bf O}_0$ and $d_E h(x)=X 1_{\bf O}$, where $X\in Lie(K)$. The subspace $Lie(K) 1_{\bf O}$ is a $Lie(G)$-submodule of ${\bf O}_0$. In fact, $d\rho(Lie(G))=Lie(G)\subseteq Lie(K)$ and $Y(X1_{\bf O})=[Y, X]1_{\bf O}\in Lie(K) 1_{\bf O}$ for
$X\in Lie(K)$ and $Y\in Lie(G)$. Consider $Z=X_{-\omega_1 +\omega_2}-X_{\omega_1 +\omega_2}\in Lie(SO(7))$.
Then $Y=d\rho(Z)=-Y_{\omega_1+\omega_2}+Y_{-\omega_3-\omega_4}\in Lie(K)$. By straightforward calculations we obtain $Y 1_{\bf O} =-{\bf u}_1$. 
Since ${\bf O}_0$ is an irreducible $Lie(G)$-module, it implies $Lie(K) 1_{\bf O}={\bf O}_0$.\end{proof}
\begin{rem}\label{OasKmodule}(see \cite{schw2}, p.629)
The $K$-module $\bf O$ is irreducible. In fact, if $x\in {\bf O}$ and $x\not\in F 1_{\bf O}$, then $Lie(G)x ={\bf O}_0$. Furthermore, $Y{\bf v}_1=-e_2$ and $1_{\bf O}=e+2e_2$. 
\end{rem}
\begin{rem}\label{reformulation}
The statement of Theorem \ref{affinestructureofV} can be reformulated as follows. The affine $SO(7)$-variety $V$ is isomorphic to ${\bf O}_1/\{\pm E\}=M$ subject to the action $g\cdot m=g^{\rho}m$ for $g\in SO(7)$ and $m\in M$. 
\end{rem}
Let $A$ denote the subalgebra $F[M]=F[{\bf O}]^{\{\pm E\}}\subseteq F[{\bf O}]$. If $z=\sum_{1\leq i\leq 8}z_i e_i$ is a generic octonion related to $\bf O$, then $A$ is generated by the elements $z_i z_j$ for $1\leq i\leq j\leq 8$. As a direct summand of $F[{\bf O}]$, $A$ is a $SO(8)$-module with a good filtration. By Lemma 3.1.3 from \cite{don}, $A$ is a $PSO(8)$-module with a good filtration as well. 

Let $A'$ denote a $PSO(8)$-algebraic algebra that coincides with $A$ as algebra but the new action of $PSO(8)$ is given by {\it twisting} with the automorphism $\rho$, i.e. $g\cdot a=g^{\rho}a$ for $g\in PSO(8)$ and $a\in A'$. By Remark \ref{saturated}, $A'$ is a $PSO(8)$-module with a good fltration. Furthermore, Remark \ref{reformulation} implies that $F[V]$ is isomorphic to $A'/ A'(n(z)-1)$ as a $SO(7)$-algebraic algebra.  
\begin{cor}\label{G_2saturated}
The subgroup $G$ is saturated in $SO(7)$.
\end{cor} 
\begin{proof} Arguing as in the proof of Lemma \ref{SO(7)issaturated}, we see that the algebra $F[V]\simeq A'/A'(n(z)-1)$ is a $PSO(8)$-module with a good filtration. Lemma \ref{SO(7)issaturated} concludes the proof. 
\end{proof}  
\begin{cor}\label{onemoreproof}
Since all $\Lambda^t({\bf O}_0)$ are tilting $SO(7)$-modules, the statements of Proposition \ref{exterioraretilting} follow.
\end{cor}
\begin{cor}\label{factorforspin}
The orbit map $K\to {\bf O}_1$ can be identified with the factor-morphism $K\to K/G$.
\end{cor}
\begin{proof}
It has been already proved in Theorem \ref{affinestructureofV} that this morphism is separable. Since every $g\in Stab_K(1_{\bf O})$ belongs to $\{\pm x | x\in G\}$, it is enough to observe that $(-x)1_{\bf O}=-1_{\bf O}$ for every $x\in G$.
\end{proof}
\begin{cor}\label{spinissaturatedin}
The subgroup $K$ is saturated in $SO(8)$, and the subgroup $G$ is saturated in $K$. 
\end{cor}
\begin{proof}
The affine variety $SO(8)/K$ is isomorphic to $PSO(8)/SO(7)^{\rho}$. Lemma \ref{SO(7)issaturated} and Remark \ref{saturated} combined with Lemma 3.1.3 from \cite{don} conclude the proof.
\end{proof}
\begin{pr}\label{Hilbertforspin}
If $char F\neq 2$, then the Hilbert-Poincare series of $F[{\bf O}^n]^K$ does not depend of $char F$.
Moreover, the dimension of each component $F[{\bf O}^n]_{k_1,\ldots , k_n}^K$ does not depend of $char F$ as well.
\end{pr}
\begin{proof}
By Corollary \ref{spinissaturatedin} all $K$-modules $\Lambda^i({\bf O})$ are modules with good a filtration. The claim follows using arguments presented in Proposition \ref{hilbert} and Remark \ref{component}. 
\end{proof}

\section{The first reduction}

The following statement can be easily derived from the {\it Frobenius reciprocity} and the {\it Tensor identity} (see \cite{jan}, Proposition I.3.4 and Proposition I.3.6; or \cite{gross}, Theorem 9.1).
\begin{pr}\label{gross}
Let $M$ be an algebraic group and $N$ be its closed subgroup. Let $V$ be a (possibly infinite-dimensional) $M$-module.
There is an isomorphism $(V\otimes F[M/N])^{M}\to V^N$, where $M$ acts by left multiplications on $M/N$ and diagonally on
$V\otimes F[M/N]$, given by the mapping
$v\otimes g\mapsto vg(1_{M}N)$ for $v\in V$ and $g\in F[M/N]$. If $V$ is an algebraic $M$-algebra, then this isomorphism is an algebra isomorphism.
\end{pr}
Proposition \ref{gross} implies $$(F[{\bf O}^n]\otimes F[{\bf O}_1])^{K}\simeq (F[{\bf O}^n]\otimes F[{\bf O}]/F[{\bf O}](n(z)-1))^{K}\simeq F[{\bf O}^n]^G .$$ 
Since $F[{\bf O}^n]^G\simeq F[t(z_1),\ldots t(z_n)]\otimes R(n)$, the problem to describe generators of $R(n)$ is equivalent to the problem 
of describing generators of $\tilde{R}(n)=F[{\bf O}^n]^G$. 

As in the proof of Lemma \ref{SO(7)issaturated}, one can show that the epimorphism $F[{\bf O}^n]\otimes F[{\bf O}]\to F[{\bf O}^n]$, given by $f(z_1, \ldots , z_n)\otimes h(z)\mapsto f(z_1 ,\ldots , z_n)h(1)$, induces an epimorphism $(F[{\bf O}^{n+1}])^K\to\tilde{R}(n)$.
\begin{pr}\label{noname}
Let $G$ be a reductive group and $M$ be a finite-dimensional $G$-module with a good filtration. If $W_1, \ldots , W_t$ are finite-dimensional $G$-modules such that $\Lambda^i(W_j)$ has a good filtration and $1\leq j\leq t, 1\leq i\leq \max \{\dim W_1, \ldots, \dim W_t\}$, then for any collection of non-negative integers $k_1, \ldots , k_t,$ 
the natural epimorphism of $G$-modules $$M\otimes W_1^{\otimes k_1}\otimes\ldots\otimes W_t^{\otimes k_t}\to M\otimes S^{k_1}(W_1)\otimes\ldots\otimes S^{k_t}(W_t)$$
induces an epimorphism $(M\otimes W_1^{\otimes k_1}\otimes\ldots\otimes W_t^{\otimes k_t})^G\to (M\otimes S^{k_1}(W_1)\otimes\ldots\otimes S^{k_t}(W_t))^G$ of vector spaces. 
\end{pr}
\begin{proof}Using arguments as in the proof of Lemma \ref{SO(7)issaturated} we obtain an acyclic complex 
$$0\to M\otimes\Lambda^{k_1}(W_1)\otimes\ldots\otimes S^{k_t}(W_t)\to\ldots\to M\otimes\Lambda^{k_1-l}(W_1)\otimes S^l(W_1)\otimes\ldots\otimes S^{k_t}(W_t)\to$$
$$\ldots\to M\otimes S^{k_1}(W_1)\otimes\ldots\otimes S^{k_t}(W_t)\to 0 .$$ Since all terms of this complex are $G$-modules with a good filtration, the map $$(M\otimes W_1\otimes S^{k_1 -1}(W_1)\otimes\ldots\otimes S^{k_t}(W_t))^G\to (M\otimes S^{k_1}(W_1)\otimes\ldots\otimes S^{k_t}(W_t))^G$$
is an epimorphism of vector spaces. By induction on $k_1+\ldots +k_t$, the epimorphism of $G$-modules
$$(M\otimes W_1)\otimes W_1^{\otimes (k_1 -1)}\otimes\ldots\otimes W_t^{\otimes k_t}\to (M\otimes W_1)\otimes S^{k_1 -1}(W_1)\otimes\ldots S^{k_t}(W_t)$$
induces the epimorphism
$$(M\otimes W_1^{\otimes k_1}\otimes\ldots\otimes W_t^{\otimes k_t})^G\to (M\otimes W_1\otimes S^{k_1 -1}(W_1)\otimes\ldots S^{k_t}(W_t))^G ,$$ 
hence the epimorphism
$$(M\otimes W_1^{\otimes k_1}\otimes\ldots\otimes W_t^{\otimes k_t})^G\to (M\otimes S^{k_1}(W_1)\otimes\ldots S^{k_t}(W_t))^G .$$\end{proof}
\begin{cor}\label{polarization}
A canonical map, that sends generic octonions $z_{k_1+\ldots +k_{i-1}+1}$, $\ldots , z_{k_1 +\ldots +k_i}$ (related to ${\bf O}^{k_i}$) to generic octonion $z_i$ (related to $i$-th summand ${\bf O}$) for $1\leq i\leq t$, induces an epimorphism
$$
F[{\bf O}^{k_1+\ldots +k_t}]_{1^{k_1+\ldots +k_t}}^K\to F[{\bf O}^t]_{k_1, \ldots , k_t}^K$$
of vector spaces.
\end{cor}
\begin{proof}
It has been already observed that all $K$-modules $\Lambda^i({\bf O})\simeq\Lambda^i({\bf O}^*)$ are modules with a good filtration. Use the canonical isomorphism $F[{\bf O}^t]\simeq S({\bf O}^*)^{\otimes t}$ and Proposition \ref{noname}.
\end{proof}
Corollary \ref{polarization} shows that it suffices to calculate multilinear invariants
$(({\bf O}^*)^{\otimes t})^K\simeq ({\bf O}^{\otimes t})^K$. Furthermore, the element $-E\in K$ acts on a vector $w=v_1\otimes\ldots\otimes v_t$ as $(-E)w=(-1)^t w$. Thus $({\bf O}^{\otimes t})^K\neq 0$ only if $t$ is an even integer.

From now on $t=2k$. Since ${\bf O}^{\otimes 2}$ has the induced structure of a $PSO(8)$-module, it is also a $K/\{\pm E\}=SO(7)^{\rho}$-module. Since ${\bf O}^{\otimes 2}\simeq S^2({\bf O})\oplus\Lambda^2({\bf O})$, the $SO(7)^{\rho}$-module ${\bf O}^{\otimes t}$ is a direct sum of submodules $V_1\otimes\ldots\otimes V_k$, where each $V_i$ is isomorphic either to $S^2({\bf O})$ or to $\Lambda^2({\bf O})$. In what follows ${\bf O}^{\otimes 2}$, and its summands $S^2({\bf O})$ and $\Lambda^2({\bf O})$, are regarded as $SO(7)$-modules via a {\it twisted} action by $g\cdot v=\rho(g)v$, where $g\in SO(7)$ and $v\in {\bf O}^{\otimes 2}$.
\begin{lm}\label{decompofSO(8)module} The $SO(8)$-module 
$S^2({\bf O}^*)\simeq S^2({\bf O})$ is tilting and isomorphic to $L(0)\oplus L(2\omega_1)$. In particular, $L(2\omega_1)=H^0(2\omega_1)=V(2\omega_1)$.
\end{lm}
\begin{proof}The module
$S^2({\bf O})$ is tilting because it is a direct summand of the tilting module ${\bf O}^{\otimes 2}$ (recall that $char F\neq 2$). The dominant weights of $S^2({\bf O})$ are $0, 2\omega_1$ and $\omega_1 +\omega_2$.
The subspace $S^2({\bf O}^*)^{SO(8)}$ is generated by $n(z)$ (cf. \cite{dp}). The isomorphism from Example \ref{anisom} sends $n(z)$ to $f=e_1e_2-\sum_{1\leq i\leq 3}{\bf u}_i{\bf v}_i$. Remark \ref{extensions} implies
$c_0(S^2({\bf O}))=1$ and the corresponding trivial submodule $L(0)=H^0(0)$ is generated
by the element $f$. 

Let $V$ denote $S^2({\bf O})/H^0(0)$.
Since $2\omega_1\geq\omega_1+\omega_2$,  we obtain $c_{2\omega_1}(V)=1$ and $c_{\omega_1+\omega_2}(V)\leq 1$. If $c_{\omega_1+\omega_2}(V)=1$, then $H^0(\omega_1+\omega_2)$ is the first member of a good filtration of $V$. The vector $e_1{\bf u}_1$ is a highest weight vector of $H^0(\omega_1+\omega_2)$ modulo $H^0(0)$. Then $R_{1+{\bf u}_1}(e_1{\bf u}_1)=
e_1{\bf u}_1 +{\bf u}_1^2$ implies  ${\bf u}_1^2\in H^0(\omega_1+\omega_2)_{2\omega_2}\neq 0$. On the other hand, $H^0(\omega_1+\omega_2)_{2\omega_2}\simeq\Lambda^2({\bf O})_{2\omega_2}=0$. This contradiction shows that $V=H^0(2\omega_1)$. 

Finally, we show that $H^0(2\omega_1)=L(2\omega_1)$. For every $i$, the elements $R_{1+{\bf u}_i}(e_1^2)=e_1^2+2e_1{\bf u}_i +{\bf u}_i^2$ and $L_{1+{\bf v}_i}(e_1^2)=e_1^2+2e_1{\bf v}_i +{\bf v}_i^2$ belong to $L(2\omega_1)$. Therefore, all $e_1{\bf u}_i, e_1{\bf v}_i$ and ${\bf u}^2_i, {\bf v}^2_i$ belong to $L(2\omega_1)$.
  
Furthermore, $L_{1+{\bf v}_i}({\bf u}^2_i)=e_2^2+2e_2{\bf u}_i+{\bf u}_i^2\in L(2\omega_1)$, hence $e_2^2, e_2{\bf u}_i\in L(2\omega_1)$ for $1\leq i\leq 3$; and symmetrically
$e_2{\bf v}_i\in L(2\omega_1)$ for $1\leq i\leq 3$.

Applying $L_{1+{\bf v}_i}$ to $e_1{\bf u}_i$ we obtain $e_1e_2+{\bf u}_i{\bf v}_i\in L(2\omega_1)$ for $1\leq i\leq 3$. Since we are working modulo $H^0(0)$, we conclude that $e_1e_2\in L(2\omega_1)$ and ${\bf u}_i{\bf v}_i\in L(2\omega_1)$ for every $i$. In other words, all vectors of weights $0$ and $\omega_1+\omega_2$ belong to $L(2\omega_1)$, hence $L(2\omega_1)=H^0(2\omega_1)$. By Remark \ref{splitness} our statement is proved.\end{proof}
\begin{pr}\label{twistedactiononsymm}
The module $S^2({\bf O})$, subject to the twisted action of $SO(7)$, is isomorphic to $\Lambda^0({\bf O}_0)\oplus \Lambda^3({\bf O}_0)$.
\end{pr}
\begin{proof} Let $D$ denote $L(2\omega_1)$ subject to the twisted action of $SO(7)$.
The module $D$, as a vector subspace of $S^2({O})$, has a basis consisting of the elements
$$e_1^2, e_2^2,  e_1 {\bf u}_i, e_1 {\bf v}_i, e_2 {\bf u}_i, e_2 {\bf v}_i, {\bf u}_i {\bf u}_j, {\bf v}_i {\bf v}_j, {\bf u}_k {\bf v}_l, e_1 e_2 +{\bf u}_i {\bf v}_i$$
for $1\leq i\leq j\leq 3$ and $1\leq k\neq l\leq 3$. Let $B^+$ denote the ("positive") Borel subgroup of $SO(7)$ that corresponds to the positive weights $\{\omega_t, \omega_u\pm\omega_v |2\leq t\leq 4, 2\leq u< v\leq 4\}$. The Lie algebra $Lie(B^+)$ is generated by the elements 
$$H_t, \ Y_{-\omega_1+\omega_t}-Y_{\omega_1+\omega_t} \text{ and } Y_{\omega_u \pm\omega_v}.$$
Let $v$ be a highest weight vector in $D$. Then $Lie(U^+)\cdot v=0$, which implies $v=\alpha e_1^2$ for $\alpha\in F\setminus 0$. Indeed, the element $v$ has the form
$$
\sum_{1\leq i\leq 3}a_i e_1{\bf u}_i +\sum_{1\leq i\leq 3}a'_i e_1{\bf v}_i +\sum_{1\leq i\leq 3}b_i e_2{\bf u}_i +\sum_{1\leq i\leq 3}b'_i e_2{\bf v}_i +\sum_{1\leq i\leq j\leq 3}c_{ij}{\bf u}_i{\bf u}_j +$$$$\sum_{1\leq i\leq j\leq 3}d_{ij}{\bf v}_i{\bf v}_j +
\sum_{1\leq i\neq j\leq 3}f_{ij}{\bf u}_i{\bf v}_j +\sum_{1\leq i\leq 3}g_i(e_1 e_2 +{\bf u}_i{\bf v}_i) +h_1 e_1^2 +h_2 e_2^2.
$$

Thus $$0=Y_{\omega_{k+1} -\omega_{l+1}}\cdot v=d\rho(Y_{\omega_{k+1} -\omega_{l+1}})v=Y_{\omega_{k+1} -\omega_{l+1}}v=$$
$$
a_k e_1{\bf u}_k-a'_l e_1 {\bf v}_l +b_k e_2{\bf u}_k-b'_l e_2 {\bf v}_l +2c_{ll}{\bf u}_l{\bf u}_k +\sum_{l< j}c_{lj}{\bf u}_k{\bf u}_j +\sum_{i< l}c_{il}{\bf u}_i{\bf u}_k
-2d_{kk}{\bf v}_k{\bf v}_l-$$
$$\sum_{k < j}d_{kj}{\bf v}_l{\bf v}_j 
-\sum_{i < k}d_{ik}{\bf v}_i{\bf v}_l + \sum_{j\neq k, l}f_{lj}{\bf u}_k{\bf v}_j -\sum_{i\neq k, l}f_{ik}{\bf u}_i{\bf v}_l +f_{lk}({\bf u}_k{\bf v}_k -{\bf u}_l{\bf v}_l)+(g_l -g_k){\bf u}_k{\bf v}_l .
$$ 
Therefore $a_l=a'_k=b_l=b'_k=c_{ll}=d_{kk}=g_l -g_k=0$ for every $1\leq k< l\leq 3$. Furthermore, if $s< l$ or $s> l$, then $c_{sl}=0$ or $c_{ls}=0$, respectively. Symmetrically, if $s< k$ or $s> k$, then $d_{sk}=0$ or $d_{ks}=0$, respectively. Analogously, $f_{ls}=f_{sk}=f_{lk}=0$ for any $s\neq k, l$.

Due to this, we can simplify the form of $v$ to 
$$v=a e_1{\bf u}_1 +a' e_1{\bf v}_3 +b e_2{\bf u}_1 +b' e_2{\bf v}_3 +c{\bf u}^2_1 +d{\bf v}_3^2 +
f{\bf u}_1{\bf v}_3 +g(3e_1 e_2 +\sum_{1\leq i\leq 3}{\bf u}_i{\bf v}_i) $$$$+h_1 e_1^2 +h_2 e_2^2.
$$ 
Since $d\rho(Y_{\omega_1+\omega_3})\cdot v=Y_{\omega_1-\omega_2}v=0$, one derives that $a=b=b'=c=f=g=h_2=0$, hence $v= a' e_1{\bf v}_3 +d {\bf v}_3^2 + h_1 e_1^2$. 
The equality 
$$(Y_{-\omega_1+\omega_4}-Y_{\omega_1+\omega_4})\cdot v=(-Y_{\omega_1+\omega_4}+Y_{-\omega_1-\omega_2})v=0$$ implies $a'=d=0$. 

Furthermore, $H_i\cdot e_1^2=e_1^2$ for every $2\leq i\leq 4$. In other words, the $SO(7)$-submodule of $D$, generated by $e_1^2$, is the simple socle of an induced submodule of $D$. 
This simple socle, which is a first member of a good filtration of $D$, is isomorphic to $L(\lambda)$, where $\lambda=2\lambda'_3
+p\mu$ for some dominant weight $\mu$. By Steinberg's tensor product theorem, $L(\lambda)\simeq L(2\lambda'_3)\otimes L(\mu)^{[p]}$ (cf. \cite{jan}, Corollary II.3.17). Comparing dimensions we conclude that $\mu=0$ and $D\simeq L(2\lambda'_3)=\Lambda^3({\bf O}_0)$.\end{proof}
\begin{pr}\label{twistedactiononexterior}
The module $\Lambda^2({\bf O})$, subject to the twisted action of $SO(7)$, is isomorphic to
${\bf O}_0\oplus\Lambda^2({\bf O}_0)$.
\end{pr}
\begin{proof}
Let $M$ denote $\Lambda^2({\bf O})$ subject to the twisted action of $SO(7)$. Let $v$ be a highest weight vector in $M$. The vector $v$ has the form 
$$
\sum_{1\leq i\leq 3}a_i e_1\wedge {\bf u}_i +\sum_{1\leq i\leq 3}a'_i e_1\wedge {\bf v}_i +\sum_{1\leq i\leq 3}b_i e_2\wedge {\bf u}_i +\sum_{1\leq i\leq 3}b'_i e_2\wedge {\bf v}_i +\sum_{1\leq i < j\leq 3}c_{ij}{\bf u}_i\wedge {\bf u}_j$$
$$+\sum_{1\leq i < j\leq 3}d_{ij}{\bf v}_i\wedge {\bf v}_j + 
\sum_{1\leq i, j\leq 3}f_{ij}{\bf u}_i\wedge {\bf v}_j  +h e_1\wedge e_2.
$$
As above, $Y_{\omega_k-\omega_l}\cdot v=0$ implies that
$$
v=a e_1\wedge {\bf u}_1 +a' e_1\wedge {\bf v}_3 +b e_2\wedge {\bf u}_1 +b' e_2\wedge {\bf v}_3 +c {\bf u}_1\wedge {\bf u}_2 +d {\bf v}_2\wedge {\bf v}_3 +f{\bf u}_1\wedge {\bf v}_3$$
$$+f'\sum_{1\leq i\leq 3}{\bf u}_i\wedge {\bf v}_i+h e_1\wedge e_2.$$
Working with the elements $Y_{\omega_k+\omega_l}$, we obtain $b=b'=c=f=0$ and $f' +h=0$. Analogously, the equations $(Y_{-\omega_1+\omega_t}-Y_{\omega_1+\omega_t})\cdot v=0$ 
for $t=2, 3, 4,$ imply $h=0$ and $a+d=0$.

Assume that $v=e_1\wedge {\bf u}_1 -{\bf v}_2\wedge {\bf v}_3$ is a highest weight vector of the first member of a good filtration of $M$. Since $H_2\cdot v=v$ and $H_3\cdot v=H_4\cdot v=0$, this submodule is isomorphic to $H^0(\lambda_1'+p\mu)$ for a suitable dominant weight $\mu$. Furthermore, the vector $v'=e_1\wedge {\bf v}_3$ is a highest weight vector of the first member of a good filtration of $M/H^0(\lambda_1'+p\mu)$. Since $H_2\cdot v' =H_3\cdot v'=v'$ and $H_4\cdot v'=0$, the corresponding submodule is isomorphic to $H^0(\lambda_2'+p\nu)$ for a suitable dominant weight $\nu$.

By Steinberg's tensor product theorem 
$$\dim L(\lambda_1'+p\mu)+\dim L(\lambda_2'+p\nu)=\dim L(\lambda_1')\dim L(\mu)^{[p]}+\dim L(\lambda_2')\dim L(\nu)^{[p]}=$$
$$7\dim L(\mu)^{[p]}+21\dim L(\nu)^{[p]}
$$
is greater than $\dim \Lambda^2({\bf O})=28$ provided $\mu\neq 0$ or $\nu\neq 0$.
Thus $\mu=\nu=0$ and $M$ has a good filtration with factors ${\bf O}_0\simeq H^0(\lambda_1')$ and $\Lambda^2({\bf O}_0)\simeq H^0(\lambda_2')$. As it has been already observed, both factors are tilting, hence $M\simeq {\bf O}_0\oplus\Lambda^2({\bf O}_0)$ by Remark \ref{wheninducedistilting}. 

Finally, if $v'$ is a highest weight of the first member of a good filtration of $M$, the arguments remain the same. The proposition is proved.
\end{proof}
For later use, let $\alpha$ denote an isomorphism $\oplus_{0\leq i\leq 3}\Lambda^i({\bf O}_0)\to {\bf O}^{\otimes 2}$ of $SO(7)$-modules, induced by isomorphisms from Proposition \ref{twistedactiononsymm} and Proposition \ref{twistedactiononexterior}. Then $\alpha^{\otimes k}$ maps, isomorphically, $\oplus_{0\leq t_1, \ldots , t_k\leq 3}\Lambda^{t_1}({\bf O}_0)\otimes\ldots\otimes\Lambda^{t_k}({\bf O}_0)$ onto ${\bf O}^{\otimes t}$, where $t=2k$. Our strategy is to calculate the generators of the vector subspace
$$(\oplus_{0\leq t_1, \ldots , t_k\leq 3}\Lambda^{t_1}({\bf O}_0)\otimes\ldots\otimes\Lambda^{t_k}({\bf O}_0))^{SO(7)}$$
and use them to calculate the invariants of $({\bf O}^{\otimes t})^{Spin(7)}$.

Observe that $\oplus_{0\leq t_1, \ldots , t_k\leq 3}\otimes_{1\leq i\leq k}\Lambda^{t_i}({\bf O}_0)$ is naturally isomorphic to a direct summand of the algebraic $SO(7)$-algebra $\Lambda({\bf O}_0^{\oplus k})$. The generators of the invariant subalgebra $\Lambda({\bf O}_0^{\oplus k})^{O(7)}$ were found in \cite{adamryb}.

We give a slightly modified description of these generators here. 
In the notations from the first section, we set 
\[e_1=\frac{a_1}{\iota\sqrt{2}}, \ e_i=a_i \text{ for } 2\leq i\leq 7 \text{ and } \iota^2=-1 .\]
Let $e_{ij}$ denote the $j$-th basic vector of the $i$-th summand of ${\bf O}_0^{\oplus k}$.
Let us define an involution $v\mapsto v'$ on the space ${\bf O}_0$ by the rule 
$e_1'=e_1, e_i'=-e_{i+3}$ for  $2\leq i\leq 4$.

The algebra $\Lambda({\bf O}_0^{\oplus k})^{O(7)}$ is generated by the elements
\[\psi^{(l)}_{rs}=\sum_{1\leq j_1 <j_2 <\ldots <j_{k-1}< j_k\leq 7}e_{r j_1}\wedge e_{r j_2}\wedge\ldots\wedge e_{r j_l}\wedge e'_{s j_1}\wedge e'_{s j_2}\wedge\ldots\wedge e'_{s j_l},\]
where $1\leq r< s\leq k, 0\leq l\leq 7$. Therefore, the space $(\Lambda({\bf O}_0^{\oplus k}))^{O(7)}$ is generated by the elements $\prod_{1\leq r< s\leq k}\psi^{(l_{rs})}_{rs}.$
Let $W$ denote a $SO(7)$-module $\otimes_{1\leq i\leq k}\Lambda^{t_i}({\bf O}_0)$.

By Proposition \ref{gross}, there is an isomorphism $(W\otimes F[O(7)/SO(7)])^{O(7)}\simeq W^{SO(7)}$.
The algebra $F[O(7)/SO(7)]$ is a two-dimensional algebra, generated by the function $\det : O(7)\to\{\pm 1\}$. Thus $(W\otimes F[O(7)/SO(7)])^{O(7)}$ can be identified with 
\[(W\oplus (W\otimes\det))^{O(7)}= W^{O(7)}\oplus (W\otimes\Lambda^7({\bf O}_0))^{O(7)}.\]
The above isomorphism is the identity map on $W^{O(7)}$ and is induced by the map $v\otimes e_1\wedge\ldots\wedge e_7\mapsto v$ on $(W\otimes\Lambda^7({\bf O}_0))^{O(7)}$. 

As a vector space, $(W\otimes\Lambda^7({\bf O}_0))^{O(7)}$ is generated by the elements 
\[(\prod_{1\leq r< s\leq k}\psi^{(l_{rs})}_{rs})(\prod_{1\leq r\leq k}\psi_{r, k+1}^{(l_{r, k+1})})\]
such that $t_i=\sum_{i < s}l_{i s}+\sum_{r< i}l_{r i}+l_{i , k+1}$ for $1\leq i\leq k$, where $\sum_{1\leq r\leq k} l_{r, k+1}=7$.

Let ${\bf l}=\{l_1, \ldots, l_a\}$ be a collection of positive integers such that $\sum_{1\leq i\leq a} l_i=7$.  Let $I_{{\bf l}}$ be a subset in $S_7$ consisting of all substitutions $\sigma\in S_7$ such that $\sigma(s)<\sigma(t)$ provided $l_1+\ldots +l_{i-1}+1\leq s< t\leq l_1+\ldots +l_i$ for some $i$. In other words, $I_{\bf l}$ is a set of representatives of the right $S_{\bf l}$-cosets, where $S_{\bf l}$ is a Young subgroup of $S_7$, which consists of all substitutions preserving the intervals $[l_1+\ldots +l_{i-1}+1, l_1+\ldots +l_i]$ 
for $1\leq i\leq a$.

For any collection ${\bf r}=\{r_1,\ldots , r_a\}$ of positive integers such that $1\leq r_1<\ldots <r_a\leq k$ let $\psi^{{\bf l}}_{{}\bf r}$ denote the element
\[\sum_{\sigma\in I_{{\bf l}}} (-1)^{\sigma}e_{r_1, \sigma(1)}\wedge\ldots\wedge e_{r_1, \sigma(l_1)}\wedge\ldots\wedge e_{r_a, \sigma(l+1+\ldots l_{a-1}+1)}\wedge\ldots\wedge e_{r_a, \sigma(7)}.\]
Then 
\[\prod_{r\leq k, \sum_r l_{r, k+1}=7}\psi_{r, k+1}^{(l_{r, k+1})}=\psi^{{\bf l}}_{{\bf r}}\wedge e'_1\wedge\ldots\wedge e'_7=\pm \psi^{{\bf l}}_{{\bf r}}\wedge e_1\wedge\ldots\wedge e_7,
\]
where ${\bf l}=\{l_{r, k+1}\neq 0 \mid 1\leq r\leq k\}$ and ${\bf r}=\{r \mid l_{r, k+1}\neq 0\}$.
Summarizing, we obtain the following lemma.
\begin{lm}\label{generatorsofsemi-inv}
A space $W^{SO(7)}$ is generated by suitable products of the generators $\psi^{(l)}_{rs}$ and of at most one generator $\psi^{{\bf l}}_{{\bf r}}$.
\end{lm}
\begin{rem}
The same arguments show that $\Lambda({\bf O}_0^{\oplus k})^{SO(7)}$, as a module over $\Lambda({\bf O}_0^{\oplus k})^{O(7)}$, is generated by the elements $\psi^{{\bf l}}_{{\bf r}}$.
\end{rem}
Finally, a trivial tensor factor $F=\Lambda^0({\bf O}_0)$ is mapped to an invariant $q(z_i, z_{i+1})$. Therefore, the induction on $t$ allows us to eliminate the trivial modules and consider only the summnads $(\otimes_{1\leq i\leq k}\Lambda^{t_i}({\bf O}_0))^{SO(7)}$, where each $0<t_i\leq 3$.

\section{The second reduction}

Let $V$ denote $\oplus_{0\leq i\leq 3}\Lambda^i({\bf O}_0)$. Define an epimorphism 
$\pi : V\otimes V\to V$ of $SO(7)$-modules, induced by multiplication, in such a way that a summand $\Lambda^i({\bf O}_0)\otimes\Lambda^j({\bf O}_0)$ is naturally mapped onto $\Lambda^{i+j}({\bf O}_0)$ for $i+j\leq 3$, and in the remaining cases the morphism $\Lambda^i({\bf O}_0)\otimes\Lambda^j({\bf O}_0)\to \Lambda^{i+j}({\bf O}_0)$ is composed with the isomorphism $\Lambda^{i+j}({\bf O}_0)\simeq\Lambda^{7-i-j}({\bf O}_0)$.
Using Proposition \ref{twistedactiononsymm} and Proposition \ref{twistedactiononexterior} we obtain a commutative diagram
\[\begin{array}{ccl}
{\bf O}^{\otimes 2}\otimes {\bf O}^{\otimes 2} & \stackrel{\alpha^{\otimes 2}}{\simeq} & V\otimes V \\
\downarrow & & \downarrow \ \pi\\
{\bf O}^{\otimes 2} & \stackrel{\alpha}{\simeq} & V
\end{array}\]
of $SO(7)$-modules. Denote by $\pi'$ the induced morphism ${\bf O}^{\otimes 4}\to {\bf O}^{\otimes 2}$ that makes the above diagram commutative. Finally, for each $i$ such that 
$1\leq i\leq k$, denote by $\pi_i(k)$ and $\pi'_i(k)$, respectively the morphisms 
\[id_{V^{\otimes (i-1)}}\otimes\pi\otimes id_{V^{\otimes (k-i)}} : V^{\otimes (i-1)}\otimes V^{\otimes 2}\otimes V^{\otimes (k-i)}\to V^{\otimes k}\]
and 
\[id_{{\bf O}^{\otimes 2(i-1)}}\otimes\pi'\otimes id_{{\bf O}^{\otimes 2(k-i)}} : {\bf O}^{\otimes 2(i-1)}\otimes {\bf O}^{\otimes 4}\otimes {\bf O}^{\otimes 2(k-i)}\to {\bf O}^{\otimes 2k},\] 
respectively. As above, there is a commutative diagram
\[\begin{array}{ccc}
{\bf O}^{\otimes 2(k+1)} & \simeq & V^{\otimes (k+1)} \\
\downarrow \pi'_i(k) & & \downarrow \pi_i(k) \\
{\bf O}^{\otimes 2k} &\simeq & V^{\otimes k}
\end{array}.\]

An invariant $f=(\prod_{1\leq r < s\leq k}\psi_{rs}^{(l_{rs})})\psi^{{\bf l}}_{{\bf r}}$ lies in that component $(\otimes_{1\leq i\leq k}\Lambda^{t_i}({\bf O}_0))^{SO(7)}$, where every $t_i$ is equal to $\sum_{r < i}l_{r i}+\sum_{i < s} l_{i s}+l_b$. Moreover, $l_b\neq 0$ if and only if $r_b=i$. 

The number of positive summands in $\sum_{r < i}l_{r i}+\sum_{i < s} l_{i s}+l_b$ is called the $i$-th {\it component of decomposition index} of $f$, and is denoted by $\mathsf{ind}_i (f)$. 
Clearly $1\leq\mathsf{ind}_i (f)\leq 3$.
The vector $(\mathsf{idn}_1(f), \ldots , \mathsf{ind}_k(f))$ is called the {\it decomposition index} of $f$, and is denoted by $\mathsf{ind}(f)$. 

A set $I$, that consists of all decomposition indexes of all (typical) multilinear invariants, can be partially ordered as follows. If $v=(v_1, \ldots , v_k)$ and $w=(w_1, \ldots , w_l)$ are two decomposition indexes, then $v < w$ if and only if $\sum_{1\leq i\leq k}v_i=\sum_{1\leq j\leq k}w_j, k\geq l$, and for the first index $i\leq l$ 
such that $v_i\neq w_i$, there is $v_i < w_i$.
It is clear that for a given $w$ the set $\{v\mid v\leq w\}$ is finite and contains a unique minimal element of the form $(1, \ldots, 1)$.
\begin{pr}\label{keyprop2}
The spinor invariants are generated by the invariants of degree $2$ and $4$ if and only if the following comditions are satisfied.
\begin{enumerate}
\item The elements $\alpha^{\otimes a}(\psi^{\bf l})$ are polynomials in invariants of degree $2$ and $4$; 
\item For arbitrary $Spin(7)$-equivariant or, equivalently $SO(7)$-equivariant, linear map $\phi : {\bf O}^{\otimes 4}\to {\bf O}^{\otimes 2}$ denote by 
$\phi_i' : {\bf O}^{\otimes 2k}\to {\bf O}^{\otimes (2k-2)}$ the induced map $id_{{\bf O}^{2(i-1)}}\otimes\phi\otimes id_{{\bf O}^{\otimes 2(k-i-1)}}$ for $3\leq k\leq 8$
and $1\leq i\leq k-1$. Then $\phi_i'$ sends every (multilinear) polynomial in the invariants of degree $2$ and $4$ to a polynomial in the same invariants.
\end{enumerate}
\end{pr}
\begin{proof}
The necessary condition is obvious.

Conversely, assume that (1) and (2) hold. We will use an induction on decomposition index of $f=(\prod_{1\leq r < s\leq k}\psi_{rs}^{(l_{rs})})\psi^{{\bf l}}_{{\bf r}}$, where ${\bf r}$ can also be the empty set. Consider an invariant $f$ that belongs to a component $(\otimes_{1\leq i\leq k}\Lambda^{t_i}({\bf O}_0))^{SO(7)}$.

If $\mathsf{ind}(f)$ is minimal, then there is partition of the index set $\{1, 2, \ldots, k\}$ into three disjoint subsets $R, S$ and ${\bf r}$, 
where $R$ and $S$ consist of the first and second indexes $r$ and $s$ of factors $\psi_{rs}^{(l_{rs})}$ respectively. Moreover, there is a bijection 
$\beta : R\to S$ such that $t_r=l_{r, \beta(r)}$ and $t_s=l_{\beta^{-1}(s), s}$ for every $r\in R, s\in S$, and $t_{r}=l_r$ for every $r\in {\bf r}$. 
It is now clear that the corresponding spinor invariant is a product of invariants of degree $4$, given by images of $\psi_{rs}^{(l_{rs})}$, and of at most one invariant that is the image of $\psi^{{\bf l}}_{{\bf r}}$. Up to a permutation of variables the latter factor is equal to $\alpha^{\otimes a}(\psi^{{\bf l}})$.

Assume that $\mathsf{ind}(f)$ is not minimal. Then  there is $i$ such that $t_i=\sum_{r < i}l_{r i}+\sum_{i < s} l_{i s}+l_b$ contains at least two positive summands. For example, assume that $0< l_{r i} < t_i$ for some $r< i$ (the other cases can be treated analogously).
There is an invariant $f'$ from 
\[(\otimes_{1\leq j < i}\Lambda^{t_j}({\bf O}_0)\otimes\Lambda^{l_{r i}}({\bf O}_0)\otimes\Lambda^{t_i-l_{r i}}({\bf O}_0)\otimes_{i < j}\Lambda^{t_j}({\bf O}_0))^{SO(7)},\]
such that $\mathsf{ind}(f') < \mathsf{ind}(f)$ and $\pi_i(k)(f')=f$. More precisely, the element $f'$ has a form $(\prod_{1\leq a < b\leq k+1}\psi_{ab}^{(l'_{ab})})\psi^{{\bf l}}_{{\bf r'}}$, where
$l'_{i b}=0$ and $l'_{a i}\neq 0$ if and only if $a=r$ and $l'_{r i}=l_{ri}$. Further, if $b< i$ ($a> i$), then $l'_{ab}=l_{ab}$ (respectively, $l'_{ab}=l_{a-1, b-1}$). On the other hand, if $a< i < b$, then
$l'_{ab}=l_{a, b-1}$. Finally, if ${\bf r}=\{r_1, \ldots, r_c\}$, then ${\bf r'}=\{r_1, \ldots , r_b+1, \ldots, r_c+1\}$, where $b$ is the minimal index among all $d$ such that $r_d\geq i$.

Let $u$ and $u'$ denote the spinor invariants which are images of $f$ and $f'$, respectively. Then $\pi'_i(k)(u')=u$. 
Furthermore, since by induction hypothesis $u'$ is a sum of products of spinor invariants of degree at most $4$, so is $u$. The proposition is proven.
\end{proof}
\begin{rem}\label{moveallto}
Up to permutation of variables (or equivalently, of tensor factors), it is enough to consider the condition (2) of Proposition \ref{keyprop2} only in the case $i=1$.
\end{rem}
Observe that for arbitrary non-negative integers $m$ and $k$, any linear map $\phi\in Hom_{Spin(7)}({\bf O}^{\otimes 2m}, {\bf O}^{\otimes 2k})$ can be identified with
an invariant $f(z_1, \ldots , z_{2(m+k)})\in F[{\bf O}^{2(m+k)}]_{1^{2(m+k)}}^{Spin(7)}\simeq ({\bf O}^{\otimes 2(m+k)})^{Spin(7)}$.

More precisely, according to Example \ref{anisom} the isomorphism $\iota : F[{\bf O}^n]\simeq S({\bf O}^n)$ is defined by $z_{i1}\mapsto e_{i 2}, z_{i2}\mapsto e_{i 1}, z_{ij}\mapsto -e_{i, j+3}, z_{i k}\mapsto -e_{i, k-3}$ for $3\leq j\leq 5$ and $6\leq k\leq 8$. Thus each element $\prod_{1\leq s\leq 2(k+m)}z_{s i_s}$ can be interpreted as a linear map from
$Hom_F({\bf O}^{\otimes 2m}, {\bf O}^{\otimes 2k})$ that takes a basic vector $e_{1 j_1}\otimes\ldots\otimes e_{2m, j_{2m}}$ to
\[z_{1 i_1}(e_{1 j_1})\ldots z_{2m, i_{2m}}(e_{2m, j_{2m}})\iota(z_{2m+1, i_{2m+1}})\otimes \ldots\otimes\iota(z_{2(k+m), i_{2(k+m)}})=\]
\[\delta_{i_1 , j_1}\ldots\delta_{i_2m, j_{2m}}\iota(z_{2m+1, i_{2m+1}})\otimes\ldots\otimes\iota(z_{2(k+m), i_{2(k+m)}}).\]

As above, for every $t\geq m$ one can define a $Spin(7)$-equivariant map
\[\phi'=\phi\otimes id_{{\bf O}^{\otimes 2(t-k)}} : {\bf O}^{\otimes 2t}\to{\bf O}^{\otimes 2(t-m+k)}.\]

Recall that the element $h$ from Lemma \ref{zeronorm} induces an involution on the space ${\bf O}$. Its matrix coincides with the matrix of the bilinear form $q$ with respect to the basis $e_i$ for $1\leq i\leq 8$. Additionally, $\iota(z_{ij})=h(e_{ij})$.
\begin{lm}\label{howtocalculateamapinducedby f}
If we identify each ${\bf O}^{\otimes 2t}$ with a multilinear component of $F[{\bf O}^{2t}]$, then $\phi'$ is defined on polynomials by the rule:
\[\phi'(u)(z_1, \ldots, z_{2(t-m+k)})=\]
\[\sum_{1\leq i_1,\ldots, i_{2m}\leq 8}f(h(e_{1 i_1}), \ldots, h(e_{2m, i_{2m}}), z_1,\ldots, z_{2k})u(e_{1 i_1}, \ldots , e_{2m, i_{2m}}, z_{2k+1}, \ldots , z_{2(t-m+k)}). 
\]
\end{lm}
\begin{proof}
We have
\[u(z_1, \ldots, z_{2t})=\sum_{1\leq i_1,\ldots , i_{2t}\leq 8}u(e_{1 i_1}, \ldots , e_{2t, i_{2t}})
\prod_{1\leq s\leq 2t}z_{s i_s}.\]
By the above remark, up to renumbering of variables, $\phi'$ maps $u$ to
\[\sum_{1\leq i_1,\ldots , i_{2t}\leq 8}u(e_{1 i_1}, \ldots , e_{2t, i_{2t}})f(h(e_{1 i_1}), \ldots, h(e_{2m, i_{2m}}), z_1,\ldots , z_{2k})\prod_{2k+1\leq s\leq 2(t-m+k)}z_{s i_s}.
\]
The statement of the lemma is now evident.
\end{proof}
Let $f$ and $u$ be two multilinear polynomials depending of common variables, say $z_{i_1}, \ldots, z_{i_l}$. 
(It is possible that degrees of $f$ or $g$ are strictly greater than $l$). The polynomial
\[\sum_{1\leq s_1, \ldots , s_l\leq 8}u|_{z_{i_u}\mapsto e_{i_u, s_u}}f|_{z_{i_u}\mapsto h(e_{i_u, s_u})}\]
is called the {\it convolution} of $f$ and $u$ on the set $\{z_{i_1}, \ldots, z_{i_l}\}$, and is denoted by $f\star_{i_1,\ldots , i_l} u$.

For example, the polynomial $\phi'(u)$ from Lemma \ref{howtocalculateamapinducedby f} is equal to
\[(f|_{z_i\mapsto y_{i-2m}, 2m+1\leq i\leq 2(m+k)}\star_{1, \ldots , 2m} u|_{z_j\mapsto y_{j-2(m-k)}, 2m+1\leq j\leq 2t})|_{y_j\mapsto z_j}.
\]

Following \cite{schw2}, one can define a multilinear invariant $Q(z_1, z_2, z_3, z_4)$ of degree $4$ as the complete skew symmetrization of the invariant $F(z_1, z_2, z_3, z_4)=t(M(z_1, z_2)M(z_3, z_4))$ with respect to its arguments, where $M(z_1, z_2)=
\frac{1}{2}(z_1 z_2 -z_2 z_1)+t(z_1)z_2-t(z_2)z_1$. To simplify our notations, let $q(i j)$ denote $q(z_i, z_j)$ and let $F(i j k l)$ (and $Q(i j k l)$, respectively) denote 
$F(z_i, z_j, z_k, z_l)$ (and $Q(z_i, z_j, z_k, z_l)$, respectively).

One can easily check that $Q(1 2 3 4)$ equals
\[\frac{1}{6}(F(1 2 3 4)-F(3 2 1 4)-F(4 2 3 1)-F(1 4 3 2)-F(1 3 2 4)-F(3 4 2 1)).\]
\begin{pr}\label{substitutions} 
For every substitution $z_i\mapsto a_i$, where $1\leq i\leq 4$ and every choice of elements $a_1, a_2, a_3, a_4$ belonging to the set $\{e_1, e_2, {\bf u}_i, {\bf v}_i\mid 1\leq i\leq 3\}$, the number $Q(1234)|_{z_i\mapsto a_i}$ belongs to $\mathbb{Z}[\frac{1}{2}]$.
\end{pr}
\begin{proof}
If every $a_i$ belongs either to $\{{\bf u}_i\mid1\leq i\leq 3\}$ or to $\{{\bf v}_i\mid1\leq i\leq 3\}$, then $Q(1234)|_{z_i\mapsto a_i}=0$.

Assume that only one $a_i$ belongs to $\{{\bf u}_i\mid1\leq i\leq 3\}$ and others belong to $\{{\bf v}_i\mid1\leq i\leq 3\}$. Without a loss of generality we can assume that $a_1={\bf u}_1, a_2={\bf v}_1, a_3={\bf v}_2$ and $a_4={\bf v}_3$. Straightforward calculations show that $Q(1234)|_{z_i\mapsto a_i}$ again vanishes.

Consider the substitution $z_1\mapsto {\bf u}_1, z_2\mapsto {\bf u}_2, z_3\mapsto {\bf v}_i, z_4\mapsto {\bf v}_j$, where $i\neq j$. Then
\[Q(1234)|_{z_i\mapsto a_i}=\frac{1}{6}(2(-1)^{\epsilon_{j i}}\delta_{k 3}+\delta_{i 2}\delta_{j 1}-\delta_{i 1}\delta_{j 2}),\]
where $\{ k\}=\{1, 2, 3\}\setminus\{i, j\}$. If $k=3$, then $\{i, j\}=\{1, 2\}$ and the above number equals $\pm\frac{1}{2}$. Otherwise it equals zero.

Furthermore, assume that $a_1=e_1, a_2=e_2$ and $a_3, a_4\in\{{\bf u}_i, {\bf v}_i\mid 1\leq i\leq 3\}$.
If both $a_3$ and $a_4$ belong to either $\{{\bf u}_i\mid 1\leq i\leq 3\}$ or $\{{\bf v}_i\mid 1\leq i\leq 3\}$, then $Q(1234)|_{z_i\mapsto a_i}=0$. If they do not, then we can assume
that $a_3={\bf u}_i, a_4={\bf v}_j$, and in this case $Q(1234)|_{z_i\mapsto a_i}=-\delta_{ij}$.

Finally, it remains to consider the case when $a_1\in\{e_1, e_2\}$ and $a_2, a_3, a_4\in\{{\bf u}_i, {\bf v}_i\mid 1\leq i\leq 3\}$. We leave for the reader to check that
\[Q(e_1, {\bf u}_1, {\bf u}_2, {\bf u}_3)=\frac{3}{2},\ Q(e_1, {\bf v}_1, {\bf v}_2, {\bf v}_3)=-\frac{1}{2},\ Q(e_1, {\bf u}_i, {\bf v}_j, {\bf v}_k)=Q(e_1, {\bf u}_i, {\bf u}_j, {\bf v}_k)=0.\]
All the remaining cases are similar to the one considered above up to the action of the element $h$ from Lemma \ref{zeronorm}.
\end{proof}
\begin{cor}\label{definedoverintegers}
The element $Q(1234)$ belongs to $\mathbb{Z}[\frac{1}{2}][z_{ij}\mid 1\leq i\leq 4, 1\leq j\leq 8]$.
\end{cor}
A (not necessary multilinear) invariant $u$ is called {\it standard} if it has a form $u=\sum u_i$, where each $u_i$ is a product of several invariants of type $q$ and $Q$. Recall that $q(ii)=q(z_i, z_i)=2n(z_i)$.
\begin{lm}\label{degreeatmost6}
Every multilinear invariant of degree at most $6$ is standard.
\end{lm}
\begin{proof}
By Proposition \ref{Hilbertforspin}, $\dim F[{\bf O}^4]_{1^4}$ and $\dim F[{\bf O}^6]_{1^6}$ do not depend on $char F$. If $char F=0$, then the space $\dim F[{\bf O}^4]_{1^4}$ is spanned by the vectors
$$q(12)q(34), q(13)q(24), q(14)q(23), Q(1234).$$
It is sufficient to prove that these vectors are linearly independent whenever $char F\neq 2$. Assume that they are not and we have the dependency relation
\[\alpha q(12)q(34)+\beta q(13)q(24)+\gamma q(14)q(23)+\delta Q(1234)=0.
\]
The substitution $z_1\mapsto e_1, z_2\mapsto {\bf v}_1, z_3\mapsto {\bf v}_2, z_3\mapsto {\bf v}_3$ takes the first three summands to zero and $Q(1234)$ to $-\frac{1}{2}$, which
implies $\delta=0$. Also, the substitution $z_i\mapsto x_i$ takes each $q(z_i, z_j)$ to $-t(x_i x_j)$. Therefore Proposition \ref{rationalinv} implies that the elements $q(12), q(13), q(23), q(14), q(24)$ and $q(34)$ are algebraically independent, which means $\alpha=\beta=\gamma=0$.

Consider a relation
\[\sum_{i_1 < i_2, i_3< i_4, i_5< i_6, i_1< i_3< i_5}\alpha_{i_1, i_2, i_3, i_4, i_5, i_6} q(i_1 i_2)q(i_3 i_4)q(i_5 i_6) \ +\]
\[\sum_{j_1< j_2 < j_3< j_4, j_5< j_6}\beta_{j_1, j_2, j_3, j_4, j_5, j_6}Q(j_1 j_2 j_3 j_4)q(j_5 j_6)=0,\]
where $\{i_1, i_2, i_3, i_4, i_5, i_6\}=\{j_1, j_2, j_3, j_4, j_5, j_6\}=\{1, 2, 3, 4, 5, 6\}$.

For a given polynomial $Q(j_1 j_2 j_3 j_4)q(j_5 j_6)$, let us consider a substitution $z_{j_5}\mapsto {\bf u}_1, z_{j_6}\mapsto {\bf v}_1, z_{j_k}\mapsto a_k$, where $\{a_1, a_2, a_3, a_4\}=\{e_1, {\bf u}_1, {\bf u}_2, {\bf u}_3\}$. If $z_{j_k}\mapsto {\bf u}_1$, then this substitution takes all polynomials from the above sum to zero except for 
$Q(j_1 j_2 j_3 j_4)q(j_5 j_6)$ and $Q( .. j_5 .. )q(j_k j_6)=\pm Q(j_1 j_2 j_3 j_4)q(j_5 j_6)|_{z_5\mapsto z_{j_k}, z_{j_k\mapsto z_5}}$.
Thus $\beta_{j_1, j_2, j_3, j_4, j_5, j_6}=(-1)^{s+1}\beta_{j'_1, j'_2, j'_3, j'_4, j'_5, j'_6}$, where $\{j_k, j_6\}=\{j'_5, j'_6\}, \{j_1', j'_2, j_3', j_4'\}=\{.., j_5, ..\}$ and $s=|\{l\mid j_5> j_l, k+1\leq l\leq 4\}|+|\{l\mid j_5< j_l, 1\leq l\leq k-1\}|$. The similar equality holds if one interchanges $z_{j_6}$ and any $z_{j_k}$ for $1\leq k\leq 4$. In other words, each $\beta_{j_1, j_2, j_3, j_4, j_5, j_6}$ equals $\pm \beta_{1, 2, 3, 4, 5, 6}$. Let $\beta$ denote $\beta_{1, 2, 3, 4, 5, 6}$.

The substitution $z_{2k-1}\mapsto e_1, z_{2k}\mapsto e_2$, where $1\leq k\leq 3$, maps all polynomials from the above sum to zero except for $q(1 2)q(3 4)q(5 6)$ which is mapped to 
$1$, hence $\alpha_{1, 2, 3, 4. 5. 6}=0$. On the other hand, the substitution $z_{2k-1}\mapsto {\bf u}_k, z_{2k}\mapsto {\bf v}_k$ for $1\leq k\leq 3$ maps
only the polynomials $Q(1234)q(56), Q(1256)q(34), Q(3456)q(12)$ and $q(12)q(34)q(56)$ to non-zero scalars. Since $Q({\bf u}_i, {\bf v}_i, {\bf u}_j, {\bf v}_j)=\frac{1}{2}$, we have$\frac{1}{2}(\beta-\beta-\beta)=-\frac{1}{2}\beta=0$. Therefore, every $\beta_{j_1, j_2, j_3, j_4, j_5, j_6}$ vanishes. Finally, applying the substitution $z_{i_{2k-1}}\mapsto {\bf u}_k, z_{i_{2k}}\mapsto {\bf v}_k$ for $1\leq k\leq 3,$ we obtain that each $\alpha_{i_1, i_2, i_3, i_4, i_5, i_6}$ equals to zero as well.
\end{proof}
\begin{lm}\label{convolutionofstandardpolynomials}
Let $f$ and $u$ be standard multilinear polynomials having $z_{i_1},\ldots , z_{i_l}$ as common variables. Then $f\star_{i_1,\ldots , i_l} u$ is again standard.
\end{lm}
\begin{proof}
Without a loss of generality one can assume that $f$ and $u$ are just products of factors of type $q$ and $Q$. Say, $f=S_1(A_1, \ldots)\ldots S_k(A_k, \ldots)$, where $A_1\sqcup\ldots\sqcup A_k=\{i_1, \ldots , i_l\}$ and $S_i\in\{q, Q\}$ for $1\leq i\leq k$. Here, the notation $S(A, \ldots)$ indicates that this polynomial depends on the variables 
$z_{i_a}$ for $a\in A$ (and additional variables masked by dots). Since
$f\star_{i_1,\ldots , i_l} u=S_k\star_{A_k}\ldots(S_1\star_{A_1} u)\ldots)$, without a loss of generality one can assume that $k=1, f=S=S_1$ and $A=A_1=\{i_1, \ldots , i_l\}$ for 
$l\leq 4$. Similar arguments allow us to assume that $u=S'(A, \ldots)$. Then $f\star_A u$ has degree at most $8-2|A|\leq 6$. Lemma \ref{degreeatmost6} concludes the proof. 
\end{proof}
\begin{lm}\label{item(2)}
All spinor invariants are standard if and only if all $\alpha^{\otimes a}(\psi^{\bf l})$ are standard.
\end{lm}
\begin{proof}
Combine Lemma \ref{convolutionofstandardpolynomials} and Proposition \ref{keyprop2}.
\end{proof}
For a vector space $W$ and two integers $0\leq i\leq k\leq\dim W$ we define a $GL(W)$-equivariant linear map $\pi_{k, i} : \Lambda^k(W)\to \Lambda^i(W)\otimes\Lambda^{k-i}(W)$ by the rule
\[w_1\wedge\ldots\wedge w_k\mapsto \sum_{\sigma\in S_k/S_i\times S_{k-i}}(-1)^{\sigma}
w_{\sigma(1)}\wedge\ldots\wedge w_{\sigma(i)}\otimes w_{\sigma(i+1)}\wedge\ldots\wedge w_{\sigma(k)}\]\[\text{ for } w_1, \ldots , w_k\in W.\]
The map $\pi_{k, i}$ is a component of comultiplication map of Hopf superalgebra $\Lambda(W)$. The reader can find more details about Hopf superalgebra structure of $\Lambda(W)$ in \cite{abw}, I.2, where
it is called a graded Hopf algebra.

Recall that $V$ denotes $\oplus_{0\leq i\leq 3}\Lambda^i({\bf O}_0)$.
Suppose that in a given collection $\bf l$ there is $l_s$ that is greater than $1$. For every integer $b$ such that $0< b < l_s$ we have a linear $SO(7)$-equivariant map 
\[\phi_{s, b}=id_V^{\otimes (s-1)}\otimes (\oplus_{i\neq l_s} id_{\Lambda^i({\bf O}_0)}\oplus\pi_{l_s, b})\otimes id_V^{\otimes (a-s)}: V^{\otimes a}\to V^{\otimes (a+1)}.\]
\begin{lm}\label{extending}
The map $\phi_{s, b}$ sends $\psi^{\bf l}$ to $\psi^{\bf l'}$, where ${\bf l}'=\{l_1, \ldots, l_{s-1}, b, l_s -b, l_{s+1},\ldots , l_a\}$.
\end{lm}
\begin{proof}
The statement follows by the coassociativity law.
\end{proof}
We say that $\psi^{\bf l'}$ is obtained from $\psi^{\bf l}$ by {\it splitting its upper parameter} $l_s$.
\begin{theorem}\label{finaldescription}
All spinor invariants are standard.
\end{theorem}
\begin{proof}
As before, we define a linear $Spin(7)$-equivariant, or equivalently $SO(7)$-equivariant linear, map $\phi'_{s, b} : {\bf O}^{\otimes a}\to {\bf O}^{\otimes (a+1)}$ that makes the diagram
\[\begin{array}{ccc}
{\bf O}^{\otimes a} & \simeq & V^{\otimes a} \\
\downarrow \phi'_{s, b} & & \downarrow \phi_{s, b} \\
{\bf O}^{\otimes (a+1)} & \simeq & V^{\otimes (a+1)}
\end{array}
\]
commutative. Then $\phi'_{s, b}$ maps $\alpha^{\otimes a}(\psi^{\bf l})$ to $\alpha^{\otimes (a+1)}(\psi^{\bf l'})$. 

Since in fact the map $\phi'_{s, b}$ is induced by some map ${\bf O}^{\otimes 2}\to {\bf O}^{\otimes 4}$, which in turn corresponds to the map $\oplus_{i\neq l_s} id_{\Lambda^i({\bf O}_0)}\oplus\pi_{l_s, b}$, $\phi'_{s, b}$ can be defined by a convolution of multilinear polynomials from $F[{\bf O}^a]_{1^a}\simeq {\bf O}^{\otimes a}$ with a standard polynomial of degree $6$. By Lemma 
\ref{convolutionofstandardpolynomials} the invariant $\alpha^{\otimes (a+1)}(\psi^{\bf l'})$ is standard whenever 
the invariant $\alpha^{\otimes a}(\psi^{\bf l})$ is standard. On the other hand, every $\psi^{\bf l}$ can be obtained from $\psi^{\{3, 3, 1\}}$ by several splittings of upper parameters. Since $\alpha^{\otimes 3}(\psi^{\{3, 3, 1\}})$ has degree $6$, it is standard. Lemma \ref{item(2)} concludes the proof.
\end{proof}
As above, let $t(ij)$ and $t(i j k)$ denote $t(x_i x_j)$ and $t((x_i x_j) x_k)$ respectively. Let $Q'(i j k l)$ denote the complete skew symmetrization of $t(((x_i x_j) x_k)x_l)$ with respect to its arguments.
\begin{cor}\label{G_2 inv}
The algebra $R(n)$ is generated by the elements 
\[t(i j), t(i j k), Q'(i j k l) \text{ for } 1\leq i, j, k, l\leq n.\]
\end{cor}
\begin{proof}
By Theorem \ref{finaldescription} and by the discussion after Proposition \ref{gross}, the algebra $R(n)$ is generated by the elements of degree at most $4$. Combining Corollary \ref{R(3)} and Proposition \ref{noname} we observe that it is sufficient to consider multilinear invariants of degree $4$ only. Since over a field of characteristic zero $R(4)_{1^4}$ is generated by the linearly independent invariants
\[t(1 2)t(3 4), t(1 3)t(2 4), t(1 4)t(2 3) \text{ and } Q'(1 2 3 4),\]
it remains to prove that they are still linearly independent over any field of odd characteristic.
Observe that $Q'(a_1, a_2, a_3, a_4)=Q(a_1, a_2, a_3, a_4)$ whenever $a_1, a_2, a_3, a_4$ are traceless octonions (cf. \cite{schw2}, (2.10)). We have 
\[Q'(e, {\bf v}_1, {\bf v}_2, {\bf v}_3)=Q(e_1, {\bf v}_1, {\bf v}_2, {\bf v}_3)-Q(e_2, {\bf v}_1, {\bf v}_2, {\bf v}_3)=-\frac{1}{2}+\frac{3}{2}=1.\]
On the other hand, the substitution $x_1\mapsto e, x_i\mapsto {\bf v}_{i-1}$ for $2\leq i\leq 4,$ maps 
the invariants
\[t(1 2)t(3 4), t(1 3)t(2 4) \text{ and } t(1 4)t(2 3)\]
to zero. Analogously to the proof of Lemma \ref{degreeatmost6}, it remains to refer to Proposition \ref{rationalinv}.
\end{proof}
\begin{rem}\label{finalrem}
Theorem \ref{finaldescription} and Corollary \ref{G_2 inv} are valid over any infinite field of odd characteristic. For example, the subgroup $G(L)$ of $G$, consisting of all $L$-points, is dense in $G$. Thus $R(n)=F[{\bf O}_0^k]^{G(L)}=F\otimes_L L[{\bf O}_0(L)]^{G(L)}$. Since the elements $t(ij)$ and $Q'(i j k l)$ are defined over $\mathbb{Z}[\frac{1}{2}]$, they are defined over $L$, too.
\end{rem}

\end{document}